\documentclass{amsart}

\usepackage{amssymb,verbatim}
\usepackage{comment}
\newcommand{\ddbar}{\sqrt{-1}\partial\overline{\partial}}
\renewcommand{\d}{\partial}
\renewcommand{\phi}{\varphi}

\newtheorem{thm}{Theorem}
\newtheorem{prop}[thm]{Proposition}
\newtheorem{lemma}[thm]{Lemma}
\newtheorem{conj}[thm]{Conjecture}

\theoremstyle{definition}
\newtheorem{defn}[thm]{Definition}

\theoremstyle{remark}
\newtheorem*{remark}{Remark}

\title{Convergence of the $J$-flow on toric manifolds}
\author{Tristan C. Collins}
\address{Department of Mathematics, Harvard University, One Oxford St., Cambridge, MA}
\email{tcollins@math.harvard.edu}
\author{G\'abor Sz\'ekelyhidi}
\address{Department of Mathematics, University of Notre Dame, 277 Hurley, South Bend, IN}
\email{gszekely@nd.edu}

\begin{document}

\begin{abstract}
  We show that on a K\"ahler manifold
  whether the $J$-flow converges or not is independent of the
  chosen background metric in its K\"ahler class. On toric
  manifolds we give a numerical characterization of when the $J$-flow
  converges, verifying a conjecture in \cite{LSz13} in this
  case. We also strengthen existing results on more general inverse
  $\sigma_k$ equations on K\"ahler manifolds. 
\end{abstract}

\maketitle

\section{Introduction}
Let $(M,\alpha)$ be a compact K\"ahler manifold of dimension $n$, and suppose that
$\Omega$ is a K\"ahler class on $M$, unrelated to $\alpha$. The $J$-flow, introduced by
Donaldson~\cite{Don99} and Chen~\cite{Chen00, Chen04} is the parabolic equation
\[ \frac{\d }{\d t}\omega_t = -\ddbar \Lambda_{\omega_t}\alpha, \]
with initial condition $\omega_0\in\Omega$. It was shown by
Song-Weinkove~\cite{SW04} that this flow converges whenever there exists
a metric $\omega\in \Omega$ satisfying $\Lambda_\omega\alpha = c$,
where the constant $c$ only depends on the classes $[\alpha], \Omega$,
and is determined by
\begin{equation} \label{eq:cdefn}
 \int_M c\omega^n - n\omega^{n-1}\wedge\alpha = 0. 
\end{equation}

In addition, Song-Weinkove~\cite{SW04} showed that such an $\omega$
exists if and only if there is a metric $\chi\in \Omega$ such that 
\begin{equation}\label{eq:BWcond}
  c\chi^{n-1} - (n-1)\chi^{n-2}\wedge\alpha > 0,
\end{equation}
in the sense of positivity of $(n-1,n-1)$-forms. Unfortunately, in
practice it seems to be almost as difficult to produce a metric $\chi$
with this positivity property, as solving the equation
$\Lambda_\omega\alpha = c$, except in the case when $n=2$, when the
condition reduces to the class $c\Omega - [\alpha]$ being K\"ahler. In
particular when $n > 2$, it was unknown whether the existence
of a solution to the equation depends only on the classes $\Omega,
[\alpha]$, or if it depends on the choice of $\alpha$ in its
class. Our first main result settles this problem by showing
that solvability depends only on the class $[\alpha]$. 

\begin{thm}\label{thm:2}
Suppose that there is a metric $\omega\in \Omega$ such that
$\Lambda_\omega\alpha = c$, and $\beta\in [\alpha]$ is another
K\"ahler metric. Then there exists an $\omega'\in \Omega$ such that
$\Lambda_{\omega'}\beta = c$. 
\end{thm}

This result still leaves open the problem of finding
 effective necessary and sufficient conditions for solvability of
the equation, or equivalently, for convergence of the $J$-flow. In this
direction, the second author and Lejmi~\cite{LSz13} proposed the following conjecture
for when the equation can be solved.

\begin{conj}\label{conj:LSz}
  There exists an $\omega\in \Omega$ satisfying $\Lambda_\omega\alpha
  = c$, with $c$ defined by Equation~\eqref{eq:cdefn}, if and only if 
  for all subvarieties $V\subset M$ with $p = \dim V < n$ we have
  \[ \int_V c\chi^p - p\chi^{p-1}\wedge \alpha > 0, \]
  for $\chi\in \Omega$. 
\end{conj}

It is natural to think of metrics $\chi$ satisfying the positivity condition \eqref{eq:BWcond}
as subsolutions for the equation $\Lambda_\omega\alpha=c$.
The result of Song-Weinkove~\cite{SW04} then
says that we can solve the equation whenever a subsolution exists,
whereas Conjecture~\ref{conj:LSz} provides a numerical 
criterion for the existence of a subsolution. In this sense it
is somewhat analogous to the result of Demailly-Paun~\cite{DP04} characterizing
the K\"ahler cone. At the same time as shown in \cite{LSz13}, the
conjecture is related to a circle of ideas in K\"ahler geometry,
relating the existence of special K\"ahler metrics to
algebro-geometric stability conditions, such as the Yau-Tian-Donaldson
conjecture~\cite{Yau93, Tian97, Don02} on the existence of constant
scalar curvature K\"ahler metrics. 

Our next result is that the conjecture holds for toric manifolds (see
Yao~\cite{Yao14} for prior results on toric manifolds). In
fact we prove the following more general result. 

\begin{thm}\label{thm:3} Let $M$ be a compact toric manifold of
  dimension $n$, with two K\"ahler metrics $\alpha, \chi$.
  Suppose that the constant $c > 0$ satisfies
  \begin{equation}\label{eq:Mint}
    \int_M c\chi^n - n\chi^{n-1} \wedge \alpha \geq 0,
  \end{equation}
  and for all toric subvarieties $V\subset M$ of dimension $p\leq n-1$
  we have
  \[ \int_V c\chi^p - p\chi^{p-1}\wedge\alpha > 0. \]
  Then there is a metric $\omega\in[\chi]$ such that
  \begin{equation}\label{eq:twisted intro}
    \Lambda_\omega\alpha + d\frac{\alpha^n}{\omega^n} = c,
  \end{equation}
  for a suitable constant $d\geq 0$. In particular if in \eqref{eq:Mint}
  we have equality, then necessarily $\Lambda_\omega\alpha = c$, and
  so  Conjecture~\ref{conj:LSz} holds for toric manifolds $M$.
\end{thm}

The advantage of this more general result is that the hypotheses are
amenable to an inductive argument. The result for $(n-1)$-dimensional
manifolds can be used to construct a suitable barrier function near
the union $D\subset M$ of the toric divisors, that
allows us to reduce the problem to obtaining a priori estimates on
compact subsets of $M\setminus D$. Theorem~\ref{thm:2} allows us to
work with torus invariant data, which on $M\setminus D$ means that the
equation reduces to an equation for convex functions on
$\mathbf{R}^n$. Equation~\ref{eq:twisted intro} is of the form 
\begin{equation}
{\rm Tr}\left((D^{2}u)^{-1}\right) + \frac{d}{\det(D^{2}u)} =1
\end{equation}
where $u:\mathbf{R}^n\to\mathbf{R}$ is convex, although we have to
deal with variable coefficients as well.  The main estimate is
an upper bound for $D^2u$ on compact sets, which we obtain in Proposition~\ref{prop:interior}.  
Using the Legendre transform and the constant rank theorem of
Bian-Guan~\cite{BG09}, this can be reduced to obtaining a priori $C^{2,\alpha}$
estimates for convex solutions of the equation
\begin{equation}\label{eq: intro model}
\Delta h + d\det(D^2 h) =1.
\end{equation}

The difficulty
is that the operator $M \mapsto Tr(M) + d \det(M)$ is neither concave,
nor convex for $d>0$,
and so the standard Evans-Krylov theory does not apply.
In addition the level sets  $\{ M: \det(M) =1-t\}$ are not uniformly convex as $t\rightarrow 1$, and so
the results of Caffarelli-Yuan~\cite{CY00} also do not apply directly.
Instead we obtain the interior $C^{2,\alpha}$
 estimates by showing that
 $\det(D^2h)^{1/n}$ is a supersolution for the linearized equation, to
 which the techniques of \cite{CY00} can be applied.  This result may be of
 independent interest.

Many of our techniques apply to more general equations than
Equation~\eqref{eq:twisted intro}, of the form
$F(A) = c$, where $A$ is the matrix $A^i_j=\alpha^{i\bar
  k}\omega_{j\bar k}$ and $F(A)$ is a symmetric function of the
eigenvalues of $A$ satisfying certain structural conditions (see
Section~\ref{sec:C2} for details). In particular we can consider
general inverse $\sigma_k$ equations of the form
\begin{equation}\label{eq:1} \sum_{k=1}^n c_k \binom{n}{k}\alpha^k\wedge \omega^{n-k}=
  c \omega^n,
\end{equation}
where $c_i \geq 0$ are given non-negative constants, and $c \geq 0$ is
determined by the $c_i$ by integrating the equation over $M$. 

The
question of looking at general equations of this form was raised by Chen~\cite{Chen00},
and some special cases beyond the $J$-flow were treated by
Fang-Lai-Ma~\cite{FLM11}, and also by Guan-Sun~\cite{GS13},
Sun~\cite{Sun13} on Hermitian manifolds. 
More general non-linear flows related to
the inverse $\sigma_k$-equations were investigated by
Fang-Lai~\cite{FL12}.
The particular Equation~\eqref{eq:twisted intro} was studied by
Zheng~\cite{Zheng14}.
In \cite{FLM11} it was shown
that a solution to these special cases of equation~\eqref{eq:1} exist if and only if there
is a metric $\chi\in\Omega$ satisfying a positivity condition
analogous to \eqref{eq:BWcond}. We show that such a result holds for
the general equation too, as was conjectured by Fang-Lai-Ma. 

\begin{thm}\label{thm:1}
  Equation~\eqref{eq:1} has a solution $\omega\in \Omega$ if and only
  if we can
  find a metric $\chi\in \Omega$ such that
  \begin{equation}\label{eq:BJpositivity}
 c\chi^{n-1} - \sum_{k=1}^{n-1} c_k\binom{n-1}{k}\chi^{n-k-1}\wedge
  \alpha^k > 0, 
\end{equation}
  in the sense of positivity of $(n-1,n-1)$-forms. 
\end{thm}

Note that when $c_k= 0$ for $k < n$, then Equation~\eqref{eq:1} is
simply a complex Monge-Amp\`ere equation, which can always be solved
by Yau's Theorem~\cite{Yau78}. In this case the positivity condition
\eqref{eq:BJpositivity} is always satisfied, so Theorem~\ref{thm:1} is
a generalization of Yau's Theorem.  
We expect that our methods can be used to generalize
Theorems~\ref{thm:2} and \ref{thm:3} to more general equations of
this form, but we will leave a detailed study of this to future work,
except for the following result that is needed in reducing
Theorem~\ref{thm:3} to the case of torus invariant $\alpha$. The proof
of this is similar to, but simpler than that of Theorem~\ref{thm:2}. 

\begin{thm}\label{thm:indep}
Suppose that Equation~\eqref{eq:1} has a solution, and $c_n > 0$. Then
the equation can also be solved if $\alpha$ is replaced by any other metric $\beta
\in [\alpha]$. 
\end{thm}

A brief summary of the contents of the paper is as follows. In
Section~\ref{sec:convex} we recall some basic convexity properties of
the elementary symmetric functions, which play a key role in the later
calculations. In Section~\ref{sec:C2} we generalize the
$C^2$-estimates of Song-Weinkove~\cite{SW04} and prove
Theorem~\ref{thm:1}. While the basic ideas are similar to those in
\cite{SW04} and also Fang-Lai-Ma~\cite{FLM11}, Fang-Lai~\cite{FL12} we hope that our
more streamlined proof highlights the required structural conditions for the
equation. In Section~\ref{sec:thm1} we prove Theorems~\ref{thm:2} and \ref{thm:indep}. A
key ingredient is a smoothing construction, based on work of
Blocki-Kolodziej~\cite{BK07} on regularizing plurisubharmonic
functions. The remainder of the paper is concerned with
the proof of Theorem~\ref{thm:3}. We expect that
our inductive method of proof will be helpful in resolving Conjecture~\ref{conj:LSz} on
non-toric manifolds as well, although there are certainly new
difficulties in the general case. 

\section{Convexity properties}\label{sec:convex}
In this section we collect some calculations relating to the inverse
$\sigma_k$-operator. The results are well known, and available in the
literature (for instance see Spruck~\cite{Spruck05}), but for the reader's
convenience we present the calculations here. 

For an $n$-tuple of numbers $\lambda_i$, denote by 
\[ S_k(\lambda_i) = \sum_{1 \leq j_1 < j_2 < \ldots < j_k \leq n}
\lambda_{j_1}\lambda_{j_2}\cdots \lambda_{j_k} \]
the elementary symmetric function of degree $k$. We have $S_0 = 1$ and
$S_{-1}=0$. For distinct indices $i_1,
\ldots, i_l$ let us write 
\[ S_{k; i_1,\ldots, i_l}(\lambda_i) = S_k(\lambda_i)|_{\lambda_{i_1}
  = \ldots = \lambda_{i_l} = 0}, \]
while if $i_1,\ldots, i_l$ are not distinct, then $S_{k,i_1,\ldots,
  i_l}(\lambda_i) = 0$. 
Given an $n\times n$ matrix $A$ we will also write $S_k(A)$ for the
elementary symmetric function of the eigenvalues of $A$, and in
addition if $A$ is diagonal then we define $S_{k; i_1,\ldots, i_l}(A)$
by letting $\lambda_i = A_{ii}$.

\begin{lemma}
  The derivatives of $S_k$ at a diagonal matrix $A$ are given by
  \[ \begin{aligned}
              \partial_{ij} S_k(A) &= \begin{cases} S_{k-1;i}(A),\text{
                  if }i=j \\
0,\text{ otherwise}  \end{cases}\\
\partial_{ij}\partial_{rs} S_k(A)&= \begin{cases} S_{k-2;i,r}(A),\text{ if }
  i=j, r=s, i\not=r \\
-S_{k-2;i,j}(A),\text{ if }i\not=j, r=j, s=i \\
0,\text{ otherwise}\end{cases}
\end{aligned}\]
Here $\partial_{ij}$ means partial derivative with respect to the
$ij$-component. 
\end{lemma}
\begin{proof}
This result follows from the fact that $S_k(A)$ is the coefficient of
$(-x)^{n-k}$ in $\det(A - xI)$. 
\end{proof}

Using this, we can compute the derivatives of the inverse $\sigma_k$
operators. 

\begin{lemma}\label{lem:dij}
For $0\leq k\leq n$ let us write $F(A) = S_k(A^{-1}) = \frac{S_{n-k}(A)}{S_n(A)}$. 
At a diagonal matrix $A$, with eigenvalues $\lambda_i$, we have
\[ \begin{aligned}
\d_{ij} F(A) &= \begin{cases} -\frac{S_{n-k;i}(A)}{\lambda_i
    S_n(A)},\text{ if }i=j \\
0,\text{ otherwise}\end{cases} \\
\d_{ii}\d_{jj}F(A) &= \frac{ S_{n-k;i,j}(A)}{\lambda_i\lambda_j
    S_n(A)},\text{ if } i\not=j, \\
 \d_{ij}\d_{ji} F(A) &= \frac{ S_{n-k;i}(A) + \lambda_i S_{n-k-1;i,j}(A)}{\lambda_i\lambda_j
  S_n(A)}, \text{ if } i\not=j, \\
\d_{ii}\d_{ii}F(A) &= 2\frac{ S_{n-k;i}(A)}{\lambda_i^2 S_n(A)},\\
\d_{ij}\d_{rs}F(A) &= 0,\text{ otherwise}.
\end{aligned}\]
\end{lemma}

The following expresses a strong convexity property of the inverse $\sigma_k$
operators on positive definite matrices. 

\begin{lemma}\label{lem:convex}
  Let us write $F(A) = S_k(A^{-1})$ as above.
  Then $A\mapsto F(A)$ is convex on the space of positive
  definite Hermitian matrices $A$, and in fact if $A$ is diagonal with
  eigenvalues $\lambda_i > 0$, then for any 
  matrix $B_{ij}$ we have
  \begin{equation}\label{eq:Fconvex}
    \sum_{p,q,r,s} B_{rs}\overline{B_{qp}} (\partial_{pq}\partial_{rs} F(A))
    + \sum_{i,j} |B_{ij}|^2 \frac{\d_{ii}F(A)}{\lambda_j} \geq 0.
    \end{equation}
\end{lemma}
\begin{proof}
Suppose that $A$ is diagonal. Given $B_{ij}$ Hermitian, using the
previous result we can compute
\[ \begin{aligned}
\sum_{p,q,r,s} B_{rs}\overline{B_{qp}} (\partial_{pq}\partial_{rs} F(A))&= \sum_{i,j}
B_{ii} \overline{B_{jj}} \frac{ S_{n-k;i,j}(A) + \delta_{ij} S_{n-k;i}(A)}{\lambda_i\lambda_j S_n(A)} \\
&\quad + \sum_{i,j} |B_{ij}|^2 \frac{ S_{n-k;i}(A) + \lambda_i
  S_{n-k-1;i,j}(A)}{ \lambda_i\lambda_j S_n(A)}. \\
&\geq \sum_{i,j}
B_{ii} \overline{B_{jj}} \frac{ S_{n-k;i,j}(A) + \delta_{ij}
  S_{n-k;i}(A)}{\lambda_i\lambda_j S_n(A)}\\
&\quad - \sum_{i,j} |B_{ij}|^2 \frac{\d_{ii} F(A)}{\lambda_j}.
\end{aligned}\]

It remains to show that the matrix $M_{ij} = S_{n-k;i,j}(A) +
\delta_{ij}S_{n-k;i}(A)$ is non-negative. This is shown in
Fang-Lai-Ma~\cite{FLM11} as follows. For any $(n-k)$-tuple
$I=\{i_1,\ldots, i_{n-k}\}\subset\{1,\ldots, n\}$, denote by $\lambda_I =
\lambda_{i_1}\cdots \lambda_{i_{n-k}}$, and let $E_I$ be the matrix whose
entries are
\[ (E_I)_{ij} = \begin{cases} 1,\text{ if } i,j\not\in I \\
  0,\text{ otherwise.}\end{cases}
\]
Then the matrix $M$ is non-negative because
\[ M = \sum_{|I|=n-k} \lambda_IE_I. \]
\end{proof}

The following is an even stronger convexity property of the map
$A\mapsto S_1(A^{-1})$, on the set of matrices with eigenvalues
bounded away from zero.

\begin{lemma}\label{lem:twistedconvex}
  Given $\delta > 0$, let $\mathcal{M}$ be the set of
  positive Hermitian matrices with eigenvalues $\lambda_i >
  \delta$. For $\epsilon > 0$ we let $F(A) = S_1(A^{-1}) - \epsilon
  S_n(A^{-1})$. If $\epsilon$ is sufficiently small (depending on
  $\delta$), then we have
  \begin{enumerate}
  \item $F(A) > 0$ and $\d_{ii}F(A) < 0$ for diagonal $A\in \mathcal{M}$.
    \item $F$ is convex on $\mathcal{M}$, and in fact $F$
  satisfies the inequality \eqref{eq:Fconvex} for diagonal
  $A\in\mathcal{M}$.
  \end{enumerate}
\end{lemma}
\begin{proof}
  \begin{enumerate}
  \item If each eigenvalue is greater than $\delta$, then $S_n(A^{-1})
    < \delta^{-(n-1)}S_1(A^{-1})$, and so $F(A) > 0$ for sufficiently
    small $\epsilon$. In addition we have
    \[ \d_{ii} F(A) = \frac{ - S_{n-1;i}(A) + \epsilon}{\lambda_i
      S_n(A)}. \]
    Since $S_{n-1;i} > \delta^{n-1}$, we have $\d_{ii} F(A) < 0$ for
    sufficiently small $\epsilon$. 
  \item
  Using the computation in the previous lemma, we have, if
  $A$ is diagonal with eigenvalues $\lambda_i$, that
  \[ \begin{aligned}
\sum_{p,q,r,s} B_{rs}\overline{B_{pq}} &(\partial_{pq}\partial_{rs} F(A))= \sum_{i,j}
B_{ii} \overline{B_{jj}} \frac{ \delta_{ij}S_{n-1;i}(A) - \epsilon(1 + \delta_{ij})}{\lambda_i\lambda_j S_n(A)} \\
&\quad + \sum_{i,j} |B_{ij}|^2 \frac{ S_{n-1;i}(A) + \lambda_iS_{n-2;i,j}(A)-\epsilon}{ \lambda_i\lambda_j S_n(A)}. \\
&\geq \frac{1}{S_n(A)}\sum_{i}
\frac{|B_{ii}|^2}{\lambda_i^2} S_{n-1;i}(A) -
\frac{\epsilon}{S_n(A)} \left|\sum_i \frac{B_{ii}}{\lambda_i}\right|^2
- \frac{\epsilon}{S_n(A)} \sum_i \frac{|B_{ii}|^2}{\lambda_i^2} \\
&\quad - \sum_{i,j} |B_{ij}|^2 \frac{\d_{ii} F(A)}{\lambda_j}.
\end{aligned}\]
If $A\in\mathcal{M}$, then $S_{n-1;i}(A) \geq \delta^{n-1}$, and so we
get the required inequality~\eqref{eq:Fconvex} if $\epsilon$ is
sufficiently small.
\end{enumerate}
\end{proof}

\section{$C^2$-estimates} \label{sec:C2}
In this section we prove an analog of the $C^2$-estimates obtained by
Song-Weinkove~\cite{SW04} for the J-flow, and
Fang-Lai-Ma~\cite{FLM11}, Fang-Lai~\cite{FL12} for
a more general class of inverse $\sigma_k$ flows. We will need the
corresponding estimates also for manifolds with boundary, analogous to
results of Guan-Sun~\cite{GS13}.

We will work with general operators
\[ F(A) = f(\lambda_1, \ldots, \lambda_n) \]
on the space of positive Hermitian matrices, where $f$ is a symmetric function of the
eigenvalues of $A$. We will require certain structural
conditions to hold for $F$. We do not expect that these conditions are
optimal, but they are sufficient for our needs. We require that there are constants $K, C
> 0$ (with $K=\infty$ allowed), and a connected component $\mathcal{M}$ of the set $\{F(A) <
K\}$, such that if $A\in \mathcal{M}$ then in coordinates such that
$A$ is diagonal, we have
\begin{enumerate}
  \item  $F(A) > 0$, and $\d_{ii}F(A) < 0$ for all $i$, 
\item For any matrix $B_{ij}$ we have
\[ \sum_{i,j,r,s} B_{ij} \overline{B_{sr}} (\d_{ij}\d_{rs}F(A)) +
\sum_{i,j} \frac{\d_{ii}F(A)}{\lambda_j} |B_{ij}|^2 \geq 0, \]
so in particular $F$ is convex on the set $\mathcal{M}$.  
\item We have
\[ C^{-1} F(A) < \sum_k -\lambda_k\d_{kk}F(A) < CF(A) \]
\item If $\lambda_1$ denotes the smallest eigenvalue, then $-\lambda_1\d_{11}F(A)
  \geq -C^{-1}\lambda_i\d_{ii}F(A)$ for all $i$. 
\item The function $g(x_1,\ldots, x_n) = f(x_1^{-1}, \ldots,
  x_n^{-1})$ extends to a smooth function on the
  orthant $\{x_i \geq 0\}$. 
\end{enumerate}

Note that in most of our situations we will be able to take
$K=\infty$, and $\mathcal{M}$ the set of all
positive Hermitian matrices. We will only need the greater generality
that we are allowing in Section~\ref{sec:thm1}.
Let us denote by $A$ the matrix $A^i_j = \alpha^{i\bar k} g_{j\bar
  k}$. This matrix is Hermitian with respect to the inner product
defined by $\alpha$. 

\begin{lemma}\label{lem:Sddbar}
   Suppose that we work at a point in normal coordinates for the
   metric $\alpha$, such that $g$ is diagonal. In addition assume that
   $A\in \mathcal{M}$. Then we have
   \[ 
\partial_1\partial_{\bar 1} F(A) \geq  -\sum_p -\d_{pp}F(A) g_{1\bar
  1}\d_p\d_{\bar p}\log g_{1\bar 1}  - CF(A), 
\]
   where $C$ depends on $F$, $\alpha$. 
\end{lemma}
\begin{proof}
  We compute, using Lemma~\ref{lem:dij} and Lemma~\ref{lem:convex}:
  \[\begin{aligned}
    \d_1\d_{\bar 1} F(A) &= \d_1( \d_{\bar 1}A^p_q \d_{pq}F(A)) \\
    &= \d_1\d_{\bar 1}(A^p_q)\d_{pq} F(A) + (\d_1 A^r_s)(\d_{\bar
      1}A^p_q) \d_{pq}\d_{rs} F(A) \\
    &= \Big[(\d_1\d_{\bar 1}\alpha^{p\bar p})g_{p\bar p} +
    \d_1\d_{\bar 1} g_{p\bar p}\Big] \d_{pp}F(A) + (\d_1 g_{s\bar r})(\d_{\bar 1}
    g_{q\bar p}) \d_{pq}\d_{rs} F(A) \\
    &= \Big[(\d_1\d_{\bar 1}\alpha^{p\bar p})g_{p\bar p} + g_{1\bar
      1}\d_p\d_{\bar p}\log g_{1\bar 1} + g^{1\bar 1}(\d_1 g_{p\bar 1})
    (\d_{\bar 1} g_{1\bar p})\Big] \d_{pp}F(A) \\
&\quad + (\d_1 g_{p\bar q})(\d_{\bar 1}
    g_{q\bar p}) \d_{pq}\d_{rs}F(A)\\
    &\geq -\sum_p - \d_{pp}F(A) g_{1\bar 1}\d_p \d_{\bar p} \log
    g_{1\bar 1}- C\sum_p -\lambda_p d_{pp}F(A),
  \end{aligned}\]
  where we used assumption (2) with $B_{ij} = \d_1 g_{i\bar j}$. 
  The result follows by assumption (3)
  above. 
\end{proof}

Suppose now that $M$ is a compact manifold
and that $\alpha, \omega_0$ are K\"ahler metrics on
$M$. Suppose that $\omega_t = \omega_0 + \ddbar \phi_t$ satisfies the equation
\begin{equation}\label{eq:parabolic}
  \begin{cases} \frac{\d \phi_t}{\d t} &= -F(A_t), \\
    \phi_0 &= 0. \end{cases}
\end{equation}
where $A_t$ is the matrix $(A_t)^i_j =
\alpha^{i\bar k}g_{t,j\bar k}$ as above, with $g_{t,j\bar k}$ being the
components of $\omega_t$. By our assumptions on $F$, this equation is
parabolic on the set of metrics satisfying $A\in \mathcal{M}$. Our
goal is to show the long time existence of this flow, generalizing the
result of Chen~\cite{Chen04} for the J-flow and Fang-Lai~\cite{FL12}
for more general flows.

The following
shows that if $A_0\in \mathcal{M}$,  then we will have $F(A_t) < K$ for as long
as the flow exists, and so in particular $A_t\in \mathcal{M}$ as
well. 

\begin{lemma}\label{lem:FAt}
  As long as $F(A_t) < K$ along the flow, we have $F(A_t) \leq
  F(A_0)$, and $|\phi_t| < C(t+1)$ for some constant
  $C$. 
\end{lemma}
\begin{proof}
  Differentiating the equation with respect to $t$ we have
  \[ \frac{\d \dot\phi_t}{\d t} = - \d_{ij}F(A_t) \alpha^{i\bar
    k}\dot{\phi}_{t, j\bar k}. \]
  The maximum principle then implies that $\inf_M \dot{\phi}_t \geq
  \inf_M \dot{\phi}_0$, i.e. $F(A_t) \leq F(A_0)$.  Similarly we can
  bound $\sup_M \dot{\phi}_t$ which in turn allows us to bound
  $\phi_t$. 
\end{proof}

To show the long time existence of the flow, we will obtain time
dependent $C^2$-estimates along the flow. 
\begin{prop} There is a constant $C$ such that
    \[ \Lambda_\alpha \omega_t < C(t+1) \]
    along the flow as long as it exists. 
\end{prop}
\begin{proof}
   Consider the function 
  \[ f(x, \xi, t) = \log |\xi|^2_{g_t} - N_1t - N_2\phi_t, \]  
  where $x\in M$ and $\xi\in T_xM$ has unit length with respect to
  $\alpha$, and $N_1, N_2$ are large constants to be chosen later. Given $T >
  0$, suppose that $f$ achieves its maximum on $[0,T]$, at $(x,\xi,
  t_0)$, with $t_0 > 0$. Choose normal coordinates at $x$ for $\alpha$
  such that $g_{t_0}$ is diagonal and $\xi = \d / \d z^1$. If we define
  \[ h(z,t) = \log\frac{g_{1\bar 1}}{\alpha_{1\bar 1}} - N_1t - N_2\phi_t, \]
  then at the origin $z=0$ and $t=t_0$ we must have
  \[ \begin{aligned}
    0 &\leq \d_t h(z,t) - \sum_p -\d_{pp}F(A_t) \d_p\d_{\bar p} h(z,t)
    \\
&= -g^{1\bar 1} F(A_t)_{1\bar 1} -N_1 - \sum_p -\d_{pp}F(A_t) \d_p\d_{\bar
  p} [\log g_{1\bar 1} - \log \alpha_{1\bar 1} - N_2\phi_t] \\
&\leq -N_1 + C\sum_p -\d_{pp}F(A_t) + N_2\sum_p -\d_{pp}F(A_t)g_{t,
  p\bar p} - N_2\sum_p -\d_{pp}F(A_t) g_{0, p\bar p}.
\end{aligned}\]
Choosing $N_2$ sufficiently large, we will have
\[ 0\leq -N_1 + N_2 \sum_p -\d_{pp} F(A_t) g_{t,p\bar p} < -N_1 + CN_2
F(A_t), \]
using property (3) of $F$. Since $F(A_t) < K$, if $N_1$ is chosen
sufficiently large, we will have a contradiction. It follows that then 
$f(x,\xi, t)$ achieves its maximum at $t=0$, from which the result
follows by the bound we already have for $\phi_t$. 
\end{proof}

Using the convexity of $F$ we can apply the Evans-Krylov theorem to 
obtain $C^{2,\alpha}$-estimates for
$\phi_t$ as long as the flow exists, and higher order estimates follow
from standard Schauder estimates. It follows that the flow exists for
all time. 

\begin{prop}\label{prop:longtime}
  There is a solution of Equation~\eqref{eq:parabolic} for all $t >
  0$. 
\end{prop}

We will make use of this long time existence result in
Section~\ref{sec:thm1}. While one could pursue the existence of solutions to
the equation $F(A) = c$ by studying the convergence of the flow, we
will instead primarily use the continuity method. For this we need elliptic
$C^2$-estimates, assuming the existence of a subsolution of the
equation in a certain sense. We will also need to use the case when
$M$ is a manifold with boundary. 

Suppose therefore that $(M,\d M)$ is
a compact K\"ahler manifold with possibly empty boundary, and suppose that
$\omega, \alpha$ are metrics on $M$ satisfying 
\[ F(A) = c \]
for a constant $c$. For simplicity we will assume that in the
structural conditions for $F$ we can take $K=\infty$. 

To define the notion of subsolution that we use, define the function
\[ \widetilde{f}(\lambda_1,\ldots, \lambda_{n-1}) =
\lim_{\lambda_n\to\infty} f(\lambda_1, \ldots, \lambda_n) =
g(\lambda_1^{-1},\ldots, \lambda_{n-1}^{-1}, 0),\]
in terms of the function $g$ in property (5). 
For a Hermitian matrix $B$ we will write
\[ \widetilde{F}(B) = \max \widetilde{f}(\lambda_1,\ldots,
\lambda_{n-1}), \]
where the max runs over all $n-1$-tuples of eigenvalues of $B$. 

\begin{remark}
  Trudinger~\cite{Tru95} studied the Dirichlet problem (over the
  reals) for equations of the eigenvalues of the Hessian satisfying
  certain structural conditions, which for example allow for
  treating the equation
  \begin{equation}\label{eq:Treq} S_1((D^2u)^{-1})= c
  \end{equation}
  in open domains $\Omega\subset\mathbf{R}^n$. Writing
  the equation as $f(\lambda_1,\ldots, \lambda_n) = c$ in terms of the
  eigenvalues of $D^2u$, a key role in the estimates is played by the
  function 
  \[ f_{\infty}(\lambda_1,\ldots, \lambda_{n-1}) = \lim_{\lambda_n
    \to\infty} f(\lambda_1,\ldots, \lambda_n), \]
  which is the same as our function $\widetilde{f}$ above (our
  equation is the reciprocal of that studied by Trudinger).  In our
  situation the function $\widetilde{f}$ is not only relevant in
  deriving estimates, but it is also used to define the notion of
  subsolution. We also remark that for the equation to fit into the
  framework of Caffarelli-Nirenberg-Spruck~\cite{CNS3}, one would
  need $f_{\infty} = 0$. 
\end{remark}

For
technical reasons we will need a notion of viscosity subsolution which
we give now. 

\begin{defn}\label{defn:viscosity}
  Suppose that $\chi$ is a K\"ahler current with continuous local
  potential, i.e. in local charts $U$ we can write $\chi = \ddbar f$
  with $f\in C^0(U)$. We say that $\chi$ satisfies 
  $\widetilde{F}(\alpha^{i\bar p}\chi_{j\bar p}) \leq c$ in the
    viscosity sense, if the following holds: suppose that $p\in M$, and $h:U\to
    \mathbf{R}$ is a $C^2$
    function on a neighborhood $U$ of $p$, where $\chi = \ddbar f$. If
    $h - f$ has a local minimum at $p$, then 
 $ \widetilde{F}(\alpha^{i\bar p}\d_j\d_{\bar p}h) \leq c$ at $p$.  
\end{defn}

It is clear from the monotonicity of $F$ (and therefore
$\widetilde{F}$), i.e. structural condition (1), that if $\chi$ is a
smooth metric satisfying $\widetilde{F}(\alpha^{i\bar p}\chi_{j\bar
  p})\leq c$, then this inequality is also satisfied in the viscosity
sense. We will need the following, which is a special case of the
general fact that a maximum of a family of viscosity subsolutions is a
viscosity subsolution (see Caffarelli-Cabr\'e~\cite{CC95}). 

\begin{lemma}\label{lem:maxviscosity}
  Suppose that in an open set $U$ we have smooth metrics
  $\chi_k = \ddbar f_k$ for $k=1,\ldots, N$, satisfying $\widetilde{F}(\alpha^{i\bar
    p}\chi_{k,j\bar p}) \leq c$. Then $\chi = \ddbar\max\{f_k\}$
  satisfies $\widetilde{F}(\alpha^{i\bar p}\chi_{j\bar p})\leq c$ in
  the viscosity sense.  
\end{lemma}
\begin{proof}
  Fix a point $p\in U$ and suppose that $h$ is a smooth function such
  that $h - \max\{f_k\}$ has a local minimum at $p$. Without loss of
  generality we can assume that $\max\{f_k(p)\} = f_1(p)$, and then
  $h-f_1$ also has a local minimum at $p$. By assumption
  $\widetilde{F}(\alpha^{i\bar p} \d_j\d_{\bar p} f_1)(p)\leq c$, and
  so the monotonicity of $\widetilde{F}$ implies that
  $\widetilde{F}(\alpha^{i\bar p}\d_j\d_{\bar p} h)(p)\leq c$. 
\end{proof}

\begin{prop} \label{prop:BenC2}
    Suppose that $\omega = \omega_0 + \ddbar\phi$ is smooth, and satisfies $F(A) =
    c$, where $A^i_j = \alpha^{i\bar p}\omega_{j\bar p}$, and $F$
    satisfies the structural conditions with $K=\infty$.  
  Suppose that we have a strict viscosity subsolution $\chi = \omega_0
  + \ddbar\psi$, i.e. that 
  $\widetilde{F}(\alpha^{i\bar p}\chi_{j\bar p})\leq c-\delta$ in the viscosity sense, for some
  $\delta > 0$. Here $\psi\in C^0(M)$. 
Then we have an estimate
\[ \Lambda_\alpha\omega < Ce^{N(\phi - \inf \phi)}, \]
where the constants $C,N$ depend on the given data, including $\chi, \delta$, as well as
the maximum of $\Lambda_\alpha\omega$ on $\d M$.
\end{prop}

\begin{proof}
  Consider the function
  \[ f(x,\xi) = \log |\xi|^2_g - N\phi(x) + N\psi(x), \]
  where $x\in M$ and $\xi\in T_xM$ has unit length with respect to
  $\alpha$, and $N$ is a large constant to be chosen later. If this
  function achieves its maximum on $\partial M$, then
  we will have
  \[ \Lambda_\alpha\omega < Ce^{N(\phi - \inf \phi)}, \]
  where $C$ depends on $\sup_{\partial M} \Lambda_\alpha\omega$ and
  the given data.

  Suppose that $f$ achieves its maximum at $(x,\xi)$, where $x\in M$
  is in the interior. We can choose normal coordinates at $x$ for $\alpha$
  such that $g$ is diagonal, and $\xi = \partial / \partial z^1$. This
  means that the function
  \[ h(z) = \log\frac{g_{1\bar 1}}{\alpha_{1\bar 1}} - N\phi + N\psi\]
  has a maximum at the origin. Define
  \[ \psi_1 = \phi - N^{-1}\log\frac{g_{1\bar 1}}{\alpha_{1\bar
      1}}. \]
  Then $\psi_1$ is smooth, and $\psi_1 - \psi$ has a local minimum at the origin. If we define
  $\chi' = \omega_0 + \ddbar \psi_1$, then the definition of viscosity
  subsolution means that $\widetilde{F}(\alpha^{i\bar p}\chi'_{j\bar
    p}) \leq c-\delta$. In addition the function 
  \[ \tilde{h}(z) = \log\frac{g_{1\bar 1}}{\alpha_{1\bar 1}} - N\phi +
  N\psi_1\]
  is identically zero.   It follows that at the origin
  \[ \begin{aligned}
           0 &= \sum_p -\d_{pp}F(A) \d_p\d_{\bar p} \widetilde{h} \\
&= \sum_p -\d_{pp}F(A) \d_p\d_{\bar p} \log g_{1\bar 1} - \sum_p
-\d_{pp}F(A) \d_p\d_{\bar p} \log \alpha_{1\bar 1} \\
&\quad - N\sum_p -\d_{pp}F(A) [g_{p\bar p} - \chi'_{p\bar p}] \\
& \geq - g^{1\bar 1}CF(A) - C\sum_p -\d_{pp}F(A) - N\sum_p
-\d_{pp}F(A)(g_{p\bar p} - \chi'_{p\bar p}),
\end{aligned} \]
where we used Lemma~\ref{lem:Sddbar}. We can assume without loss of
generality that $g_{1\bar 1} > 1$. By assumption (3) on $F$, we know
that $\sum_p -\d_{pp}F(A)g_{p\bar p}$ is bounded below. In addition
$\chi'_{p\bar p}$ also has a fixed lower bound by the assumption that
$\chi$ is a K\"ahler current. It follows that for any
$\epsilon > 0$ we can choose $N$ so large, that we obtain
\[ 0 \geq -N(1+\epsilon)\sum_p -\d_{pp}F(A) g_{p\bar p} +
N(1-\epsilon) \sum_p -\d_{pp}F(A)\chi'_{p\bar p}.\]
Rearranging this and changing $\epsilon$ slightly, for sufficiently
large $N$ we will have
\[ \sum_p -\d_{pp}F(A)\chi'_{p\bar p} \leq (1+\epsilon)\sum_p
-\d_{pp}F(A) g_{p\bar p}. \]

Let us now change notation slightly. Write $\lambda_i$ for the
eigenvalues of $g$, and $\mu_i$ for the eigenvalues of $\chi'$. In
addition suppose that $\lambda_1\leq\lambda_2\leq\ldots
\leq\lambda_n$. 
 For simplicity of notation let us also suppose that $c=1$. We have
the following.

\begin{equation}\label{eq:3} \begin{gathered}
f(\lambda_1,\ldots, \lambda_n) = 1, \\
\tilde{f}(\mu_1,\ldots, \mu_{n-1}) \leq 1-\delta,\\
\sum_p -f_p(\lambda_1,\ldots,\lambda_n)\mu_p \leq (1+\epsilon)\sum_p
-f_p(\lambda_1,\ldots, \lambda_n)\lambda_p.
\end{gathered}
\end{equation}

Using assumption (3) for $F$ and the lower bound for the $\mu_i$, we
obtain an upper bound for $-f_1(\lambda_i)$. Together with assumption
(4) for $F$, this implies a lower bound for the lowest eigenvalue
$\lambda_1$. 

By assumption (5), there is a number $K$
such that if $\lambda_n > K$, then 
\[ \begin{gathered} 
  \tilde{f}(\lambda_1,\ldots, \lambda_{n-1}) > 1-\tau, \\
  \tilde{f}_p(\lambda_1,\ldots, \lambda_{n-1}) < f_p(\lambda_1,\ldots,
  \lambda_n) + \tau,
\end{gathered}
\]
since $f_p(\lambda_i) = -\lambda_p^{-2}g_p(\lambda_i^{-1})$, and the function
$x_p^2g(x_i)$ is uniformly continuous on compact subsets of the
orthant $\{ x_i \geq 0\}$. Note that by convexity, $\tilde{f}_p \geq f_p$. 
The convexity of the map $\lambda_i \to \tilde{f}(\lambda_i)$ implies
that if we denote
\[ r(t) = \tilde{f}(\lambda_i + t(\mu_i-\lambda_i)), \]
then $ r(1) \geq r(0) + r'(0)$. 
This means
\[ \tilde{f}(\mu_i) \geq \tilde{f}(\lambda_i) + \sum_{p=1}^{n-1}
(\mu_p - \lambda_p)\tilde{f}_p(\lambda_i).\]
From this we get
\[ \delta -\tau\leq \sum_{p=1}^{n-1} -\tilde{f}_p (\mu_p -\lambda_p)
\leq \sum_{p=1}^n -f_p(\lambda_i) \mu_p - \sum_{p=1}^{n-1}
-\tilde{f}_p(\lambda_i) \lambda_p . \]
From assumption (5) we also have that $\lambda_p f_p \to 0$ uniformly as
$\lambda_p \to\infty$, since $\lambda_p f_p(\lambda_i) =
-\lambda_p^{-1}g_p(\lambda_i^{-1})$, and $x_p g(x_i)$ is uniformly
continuous on compact subsets of $\{x_i \geq 0\}$. So we can choose a
constant $K'$ such that $-f_p\lambda_p < \tau$ if $\lambda_p > K'$. In addition we
take $K$ above so that $K > K'$. We then have
\[ \begin{aligned}
  \delta - \tau &\leq \sum_{p=1}^n -f_p\mu_p - \sum_{p=1}^n
  -f_p\lambda_p + n\tau + \sum_{p=1}^{n-1} (\tilde{f}_p - f_p)K' \\
  &\leq \epsilon\sum_{p=1}^n -f_p \lambda_p + n\tau + (n-1)\tau K' \\
  &\leq C\epsilon F(A) + n\tau + (n-1)\tau K'. 
\end{aligned}\]
We can choose $\tau$ so small (i.e. $K$ above so large), that
\[ (n+1)\tau + (n-1)\tau K' < \frac{\delta}{2}. \]
We will then have
\[ \frac{\delta}{2} \leq C\epsilon F(A). \]
If now $\epsilon$ is sufficiently small (i.e. the constant $N$ above
is chosen sufficiently large), then this will be a contradiction.

It follows that if the constant $N$ before is
chosen sufficiently large, then at the maximum of our function $f$ we
have a bound $g_{1\bar 1} < K$ for some large $K$. From this it follows that we
have an inequality of the form
\[ \Lambda_\alpha g < C e^{N(\phi - \inf\phi)}, \]
which is what we wanted to prove. 
\end{proof}

We now prove Theorem~\ref{thm:1} under the more general assumption
that $\chi$ satisfies the positivity condition \eqref{eq:BJpositivity}
in the viscosity sense. For this let us write  
\begin{equation}\label{eq:Fdefn}
    F(A) = \sum_{k=1}^n c_k S_k(A^{-1}),
\end{equation}
for constants $c_k\geq 0$. Note that if $\chi$ is smooth and $B^i_j =
\alpha^{i\bar p}\chi_{j\bar p}$ then $\widetilde{F}(B) \leq c-\delta$
is equivalent to the positivity of $(n-1,n-1)$ forms
\begin{equation}\label{eq:BenJian2}
    c\chi^{n-1} - \sum_{k=1}^{n-1} c_k\binom{n-1}{k}\chi^{n-k-1}\wedge
    \alpha^k > 0.
\end{equation}
In particular, Theorem~\ref{thm:1} follows from the following.

\begin{thm} \label{thm:1visc} 
  Suppose that we have a K\"ahler current $\chi\in \Omega$
  satisfying $\widetilde{F}(\alpha^{i\bar p}\chi_{j\bar p})\leq c-\delta$ in the
  viscosity sense, for some $\delta > 0$. Suppose that the constant
  $c$ satisfies 
  \[ \sum_{k=1}^n c_k\binom{n}{k}\int_M \alpha^k\wedge \omega_0^{n-k} =
  c\int_M \omega_0^n. \]

  Then there is an
  $\omega\in\Omega$ satisfying the equation $F(\alpha^{i\bar
    p}\omega_{j\bar p}) = c$, i.e. 
  \begin{equation}\label{eq:Jeq2} \sum_{k=1}^n c_k\binom{n}{k} \alpha^k\wedge\omega^{n-k} =
  c\omega^n. \end{equation}
\end{thm}
\begin{proof} 
  We will use the continuity method to solve the equation
  \begin{equation}\label{eq:cont}
    F_d(A) = F(A) + d \frac{\alpha^n}{\omega^n} = c_d
  \end{equation}
  for $d\in [0,\infty)$, where the constant $c_d$ is determined by
  $d$ by integrating the equation with respect to $\omega^n$ over
  $M$. In particular $c_d \geq c$.  
  According to Lemma~\ref{lem:Fd} below, $F_d(A)$ satisfies the structural
  conditions required by the $C^2$-estimates (with $K=\infty$ and
  $\mathcal{M}$ the space of all positive Hermitian matrices). In addition
  $\tilde{F}_d = \tilde{F}$, so $\chi$ is a strict viscosity
  subsolution for the equation $F_d(A) = c_d$, for all $d\geq 0$. We
  will therefore be able to  use Proposition~\ref{prop:BenC2} to
  obtain $C^2$-estimates.

  Let $I = \{ d\in [0,\infty)\,:\, \text{ \eqref{eq:cont}
    has a solution}\}$. By Yau's theorem~\cite{Yau78} we can solve the
  equation $\alpha^n / \omega^n = c_\infty$ for a suitable constant
  $c_\infty$. The implicit
  function theorem then implies that we can solve $\eqref{eq:cont}$
  for sufficiently large $d$, and that $I$ is open.

  To see that $I$ is closed, suppose that $\omega_k = \omega_0 +
  \ddbar\phi_k$ are solutions, with corresponding $d_k\to
  d$. Proposition~\ref{prop:BenC2} then implies that
  \[ \Lambda_\alpha\omega_k < Ce^{N(\phi_k - \inf \phi_k)}, \]
  for uniform $C, N$. Normalizing so that $\sup\phi_k =0$,
  it follows from Weinkove~\cite[Lemma 3.4]{W06} (see also
  \cite[Proposition 4.2]{Wei04}), that we have $\Lambda_\alpha\omega_k
  < C$ for a  uniform $C$. The equation then implies a lower bound $\omega_k >
  C^{-1}\alpha$ as well. Since $F_d$ is convex, the Evans-Krylov
  theorem~\cite{Ev82, Kry82} (see
  Tosatti-Wang-Weinkove-Yang~\cite{TWWY14} for a general version
  adapted to complex geometry),
  together with Schauder estimates can be used to obtain
  higher order estimates for the $\omega_k$, allowing us to pass to a
  limit as $k\to\infty$. This shows that $I$ is closed, and so $0\in
  I$, which was our goal.  
\end{proof}

\begin{lemma}\label{lem:Fd}
  The map $F(A)$ in Equation~\eqref{eq:Fdefn} satisfies the
  structural conditions at the beginning of Section~\ref{sec:C2}, on
  the whole space of positive Hermitian matrices. 
\end{lemma}
\begin{proof}
  It is clear that $F(A) > 0$, and from Lemma~\ref{lem:dij} we have,
  at a diagonal $A$ with eigenvalues $\lambda_i >0$, that
  \[ \d_{ii}F(A) = -\sum_{k=1}^n c_k \frac{S_{n-k;i}(A)}{\lambda_i
    S_n(A)}, \]
  which is negative. The required convexity property (2) follows from
  Lemma~\ref{lem:convex}.

  For property (3), note that
  \[ \sum_i -\lambda_i\d_{ii}F(A) = \sum_{k,i} c_k
  \frac{S_{n-k;i}(A)}{S_n(A)} = \sum_{k=1}^n kc_k
  \frac{S_{n-k}(A)}{S_n(A)} = \sum_{k=1}^n kc_k S_k(A^{-1}), \]
  using the identity $\sum_{i=1}^n S_{l;i}(A) = (n-l)S_l(A)$. It
  follows that
  \[ F(A) \leq -\sum_i \lambda_i \d_{ii}F(A) \leq nF(A). \]

  To show property (4), recall that
  \[ -\lambda_i \d_{ii}F(A) = \sum_{k=1}^n c_k
  \frac{S_{n-k,i}(A)}{S_n(A)}, \]
  and note that if $\lambda_1$ is the smallest eigenvalue, then
  $S_{n-k;1}(A) \geq S_{n-k;i}(A)$ for all $i$. It follows that
  $-\lambda_1\d_{11}F(A) \geq -\lambda_i\d_{ii}F(A)$.

  Finally,
  property (5) is clear since we have
  \[ g(x_1,\ldots,x_n) = f(x_1^{-1},\ldots, x_n^{-1}) = \sum_{k=1}^n
  c_k S_k(x_1,\ldots, x_n), \]
  which extends smoothly to the orthant $\{x_i\geq 0\}$. 
\end{proof}

\section{The proof of Theorem~\ref{thm:2}}\label{sec:thm1}
In the proof of Theorem~\ref{thm:2} a key role is played by the
parabolic equation \eqref{eq:parabolic}. In this section we will
consider the operator
\[ F_\epsilon(A) = S_1(A^{-1}) - \epsilon S_n(A^{-1}), \]
for small $\epsilon > 0$. Note that this operator
is not convex on the space of all Hermitian matrices. The following lemma shows that it is still convex,
however, on a suitable set of matrices $A$.
\begin{lemma}\label{lem:QK}
 \begin{enumerate}
    \item For all $Q > 0$, if
  $\epsilon$ is sufficiently small, then $F_\epsilon$ satisfies the
  structural conditions in Section~\ref{sec:C2} on the set where $S_1(A^{-1})
  < Q$.
  \item For all $K
  > 0$, if $\epsilon$ is sufficiently small, then $F_\epsilon$
  satisfies the structural conditions in
  Section~\ref{sec:C2} on the connected component $\mathcal{M}$ of the set
  $\{F_\epsilon(A) < K\}$ containing the set $\{S(A^{-1}) < K\}$.
  \end{enumerate}
\end{lemma}
\begin{proof}
  For statement (1), note that if $S_1(A^{-1}) < Q$, then all
  eigenvalues of $A$ are greater than $1/Q$. By Lemma~\ref{lem:twistedconvex}, for
  sufficiently small $\epsilon$, the map $F_\epsilon$ satisfies the
  structural conditions (1) and (2) on the set of matrices with eigenvalues
  greater than $1/Q$. For structural condition (3), we have 
  \[ \sum_k -\lambda_k \d_{kk}F_\epsilon(A) = \sum_k \frac{
    S_{n-1;k}(A) - \epsilon}{S_n(A)} = \frac{S_{n-1}(A)}{S_n(A)} -
  \frac{n\epsilon}{S_n(A)} = S_1(A^{-1}) - n\epsilon
  S_n(A^{-1}).\]
  The required inequality follows since $S_n(A^{-1}) <
  Q^{n-1}S_1(A^{-1})$, and so if $\epsilon$ is sufficiently small, we
  will have
  \[ \frac{1}{2} F_\epsilon(A) < \sum_k -\lambda_k \d_{kk}F(A) <
  CF(A). \]

  Although we do not actually need structural assumptions (4) and (5)
  below (since we will only use Proposition~\ref{prop:longtime} which
  does not use them), they are also easy to check, (5) being
  immediate. For (4), note that
  \[ -\lambda_i\d_{ii}F_\epsilon(A) = \frac{S_{n-1;i}(A) -
    \epsilon}{S_n(A)}, \]
  and so for sufficiently small $\epsilon$ we will have
  \[ \frac{1}{2}(-\lambda_i\d_{ii}S_1(A^{-1})) \leq
  -\lambda_i\d_{ii}F_\epsilon(A) \leq \lambda_i\d_{ii}S_1(A^{-1}), \]
  so (4) follows from the corresponding property of the map $A\mapsto
  S_1(A^{-1})$, which we have shown in Lemma~\ref{lem:Fd}. 
  
  For statement (2) let us take $Q = K+1$, and take $\epsilon$
  sufficiently small for (1) to apply. In addition, note that the
  AM-GM inequality implies that $S_n(A^{-1}) \leq n^{-n}
  S_1(A^{-1})^n$, so if $S_1(A^{-1}) = K+1$, then $S_n(A^{-1})
  < K'$ for some constant $K'$. Let us choose $\epsilon$ even smaller,
  so that $K+1- \epsilon K' > K$. This means that
  if $S_1(A^{-1}) = K+1$, then
  \[ F_\epsilon(A) \geq K + 1 - \epsilon K' > K. \]
  Suppose now that $A_1$ is in the connected component of
  $\{F_\epsilon(A) < K\}$ containing $\{S_1(A^{-1}) < K\}$, i.e. we have
  positive Hermitian matrices $A_t$ for $t\in [0,1]$ such that
  $F_\epsilon(A_t) < K$, and $S_1(A_0^{-1}) < K$. Then from the above we have
  $S_1(A_t^{-1}) < K+1$ for all $t$, and in particular $S_1(A_1^{-1}) < K+1$. It
  follows from (1) that $F_\epsilon$ satisfies the structural conditions on
  this connected component. 
\end{proof}

Consider now the evolution equation
\begin{equation}\label{eq:flow}
  \begin{cases} \frac{\d\phi_t}{\d t} &= c_\epsilon - F_\epsilon(A_t) \\
  \phi_0 &= 0,
\end{cases}
\end{equation}
where $\omega_t = \omega_0 + \ddbar\phi_t$, and $(A_t)^i_j =
\alpha^{i\bar k}g_{t, j\bar k}$. In addition the constant $c_\epsilon$
is chosen so that
\[ \int_M c_\epsilon\omega^n = \int_M F_\epsilon(A)\,\omega^n.\]

From the previous lemma we have that
if $F(A_0) < K$, and $\epsilon$ is chosen sufficiently small,
then $F_\epsilon$ satisfies the structural conditions on the connected
component of $\{F_\epsilon(A) < K\}$ containing $A_0$.  
Proposition~\ref{prop:longtime} then implies that the flow exists for all
time. Note that in addition the proof of Lemma~\ref{lem:QK} also shows
that we can assume $S_1(A_t^{-1}) < K+1$ along the flow.

The parabolic equation~\eqref{eq:flow} is the negative gradient flow of
the functional $\mathcal{J}_\epsilon$ 
on the K\"ahler class $[\omega]$, defined by the variational formula
\[ \left.\frac{d}{dt}\right|_{t=0} \mathcal{J}_\epsilon(\omega + t\ddbar\phi) =
\int_M \phi(F_\epsilon(A) - c_\epsilon) \omega^n, \]
and normalized so that $\mathcal{J}_\epsilon(\omega_0)=0$ for a fixed 
choice of $\omega_0$. Note that $\mathcal{J}_0$ is the $J$-functional
considered in Song-Weinkove~\cite{SW04}. 

\begin{defn}
  We say that the function $\mathcal{J}_0$ is proper, if there
  are constants $C, \delta > 0$ such that if $\omega = \omega_0 +
  \ddbar\phi$, then 
  \[ \begin{aligned} \mathcal{J}_0(\omega) &\geq -C + \delta\int_M \phi(\omega_0^n -
  \omega^n) \\
&= -C + \delta\int_M \sqrt{-1}
  \d\phi\wedge \bar{\d}\phi\wedge(\omega_0^{n-1} +\ldots +
  \omega^{n-1}) 
\end{aligned}\]
  for all $\omega \in \Omega$. 
\end{defn}

The main ingredient in the proof of Theorem~\ref{thm:2} is the
following.
\begin{prop}\label{prop:properconverge}
  Suppose that $\mathcal{J}_0$ is proper. Then we can find $\omega\in
  \Omega$ such that $A^i_j = \alpha^{i\bar k}\omega_{j\bar k}$
  satisfies $S_1(A^{-1}) = c$, i.e. $\Lambda_\alpha\omega = c$.  
\end{prop}
\begin{proof}
  For simplicity of notation let us normalize the class $\Omega$ and $\alpha$ so that $\int_M
  \omega^n = \int_M \alpha^n = 1$. It follows that
  \[ c_\epsilon = c - \epsilon \int_M \frac{\alpha^n}{\omega^n}
  \omega^n = c-\epsilon, \]
  for any $\epsilon > 0$. If $\omega = \omega_0 + \ddbar\phi$, then we have
  \[ (\mathcal{J}_\epsilon - \mathcal{J}_0)(\omega) = \epsilon \int_0^1 \int_M
  \phi ( \omega_t^n - \alpha^n), \]
  where $\omega_t = \omega_0 + t\ddbar\phi$. Using Yau's
  Theorem~\cite{Yau78}, we can assume that we
  chose our base point $\omega_0$ so that $\omega_0^n = \alpha^n$, so
  \[ \begin{aligned}
    \mathcal{J}_\epsilon(\omega) &= \mathcal{J}_0(\omega) + \epsilon\int_0^1
    \int_M \phi(\omega_t^n - \omega_0^n) 
  \\ &\geq -C + \delta\int_M \sqrt{-1}\d\phi\wedge
  \bar{\d}\phi\wedge(\omega_0^{n-1} + \ldots + \omega^{n-1}) \\
&\quad  -\epsilon\int_0^1 \int_M \sqrt{-1}\d\phi
  \wedge\bar{\d}\phi\wedge (\omega_t^{n-1} + \ldots +
  \omega_0^{n-1}). 
  \end{aligned}\]
It follows that if $\epsilon$ is sufficiently small, then
$\mathcal{J}_\epsilon$ will be proper. 

Choose $\epsilon$ even smaller if necessary so that
Equation~\eqref{eq:flow} has a solution $\omega_t = \omega_0 +
\ddbar\phi_t$ for all $t > 0$. Since this flow is the negative
gradient flow of $\mathcal{J}_\epsilon$, and this functional is
bounded from below, we can find a sequence of metrics $\omega_k$ along
the flow such that 
\[ \lim_{k\to\infty} \int_M (F_\epsilon(A_k) -
c_\epsilon)^2\,\omega_k^n = 0. \]
Since we have $S_1(A_t^{-1}) < K+1$ along the flow for a uniform
constant $K$, we have a uniform lower bound $\omega_k >
\kappa\alpha$. It follows that
\begin{equation}\label{eq:limk} \lim_{k\to\infty} \Vert F_\epsilon(A_k) -
c_\epsilon\Vert_{L^2(\alpha)} = 0, 
\end{equation}
and from Lemma~\ref{lem:FAt} we know that $F_\epsilon(A_k) < K$. Choosing
$\epsilon$ even smaller, using Lemma~\ref{lem:twistedconvex} we can
assume that $F_\epsilon$ satisfies the structural conditions on the set of matrices with
eigenvalues bounded below by $\frac{\kappa}{2}$. The importance of
this is that this is a convex set. 

Let us write $\psi_k = \phi_k - \sup_M\phi_k$, so that $\omega_k =
\omega_0 +\ddbar\psi_k$, and $\sup_M\psi_k = 0$.
The properness of $\mathcal{J}_\epsilon$ implies that we
have a uniform constant $C$ such that
\[ \int_M \psi_k(\omega_0^n - \omega_k^n) < C. \]
A standard argument using the inequality $\omega_0 + \ddbar\psi_k > 0$
together with $\sup_M\psi_k =0$ implies that $\int_M \psi_k\omega_0^n$
is bounded below uniformly. It follows that we have a uniform bound
\begin{equation}\label{eq:Epsibound}
  -\int_M \psi_k \omega_k^n < C.
  \end{equation}
Choosing a subsequence we can assume that the $\omega_k$ converge
weakly to a current $\omega_0 + \ddbar\psi_\infty$, and $\phi_k \to
\phi_\infty$ in $L^1$. From Guedj-Zeriahi~\cite[Corollary 1.8, Corollary 2.7]{GZ07},
the bound \eqref{eq:Epsibound} implies that $\psi_\infty$ has zero
Lelong numbers.

We now use the technique of Blocki-Kolodziej~\cite{BK07} to mollify the
metrics $\omega_k$, in order to obtain pointwise bounds from the
integral bound \eqref{eq:limk}. The fact that $\psi_\infty$ has zero
Lelong numbers will ensure that we can perform this mollification
uniformly in $k$.

Fix a small number $\tau > 0$, and choose a finite open cover
$\{W_i\}$ of $M$ such that on each $W_i$ we have local coordinates
$z^j$, in which the matrix of components of $\alpha$ satisfies
\begin{equation}\label{eq:alphacomp}
  (1-\tau) \delta_{ij} < \alpha_{i\bar j} < (1 + \tau)\delta_{ij}.
\end{equation}
Let $V_i\subset U_i\subset W_i$ be relatively compact so that the
$V_i$ still cover $M$. On each $W_i$ we have $\omega_0 = \ddbar f_i$
for local potentials $f_i$, and so we have the plurisubharmonic
functions $u^{(k)}_i = f_i + \psi_k$, which are local potentials for
the $\omega_k$. We allow $k=\infty$ here. For sufficiently small
$\delta > 0$ (depending on the the distance between the boundaries of
$U_i, W_i$, and so on $\tau$) we can define plurisubharmonic functions
$u^{(k)}_{i,\delta}$ on $U_i$ by
\[ u^{(k)}_{i,\delta}(z) = \int_{\mathbf{C}^n} u^{(k)}_i(z-\delta
w)\rho(w)\,dw, \]
where $\rho:\mathbf{C}^n\to\mathbf{R}$ is a standard mollifier:
$\rho\geq 0$, $\rho(w) = 0$ for $|w|> 1$, and $\int \rho(w)\,dw = 1$.

For each $i$ we choose $\eta_i : U_i\to\mathbf{R}$ such that
$\eta_i\leq 0$, and in addition $\eta_i=0$ on $V_i$ and $\eta_i=-1$ on
$\d U_i$. Fix $\gamma > 0$ to be sufficiently small
(depending on $\tau$)
such that $\gamma|\ddbar \eta_i| < \tau\alpha$ on $U_i$ for all $i$.

Consider the function
\begin{equation}\label{eq:psik}
 \psi_{k,\delta}(z) = \max_i \{ u^{(k)}_{i,\delta}(z) - f_i(z) +
\gamma\eta_i(z) \}, 
\end{equation}
where the maximum is taken over all $i$ for which $z\in U_i$. The
results in Blocki-Kolodziej~\cite{BK07} (see the proof of Theorem 2)
imply that if $\delta$ is sufficiently small (depending on $\tau$),
then
\[ |(u_{i,\delta}^{(\infty)}-f_i) - (u^{(\infty)}_{j,\delta} - f_j)| <
\frac{\gamma}{4} \]
on $U_i\cap U_j$, since $\psi_\infty$ has zero Lelong numbers. The
$L^1$-convergence $\psi_k\to\psi_\infty$ implies that we have uniform
convergence $u^{(k)}_{i,\delta}\to u^{(\infty)}_{i,\delta}$ of the
mollifications as $k\to\infty$, and so once $k$ is chosen sufficiently
large we will have
\[ |(u^{(k)}_{i,\delta} - f_i) - (u^{(k)}_{j,\delta} - f_j)| <
\frac{\gamma}{2}, \]
on the set $U_i\cap U_j$. This implies that if $z\in \d U_i\cap V_j$,
then
\[ (u^{(k)}_{i,\delta} - f_i + \gamma\eta_i)(z) < (u^{(k)}_{j,\delta}
- f_j + \gamma\eta_j)(z), \]
so all $z\in M$ have a neighborhood $U$ such that in the definition of
$\psi_{k,\delta}$ the maximum can be taken over those $j$ for which 
$u^{(k)}_{j,\delta} - f_j + \gamma\eta_j$ is defined on $U$ (i.e. for
which $U\subset U_j$). In particular $\psi_{k,\delta}$ is
continuous. We will now see that in addition for large $k$ and small
$\delta$, the form $\chi = \omega_0 + \ddbar\psi_{k,\delta}$ satisfies
$\widetilde{F}(\alpha^{i\bar p}\chi_{j\bar p}) \leq c - \delta'$ in
the viscosity sense for small $\delta' > 0$, and so 
by Theorem~\ref{thm:1visc} we can solve the
equation $\Lambda_\omega\alpha = c$. Using
Lemma~\ref{lem:maxviscosity} it is enough to show that for large $k$
and small $\delta$ the metric $\chi_i = \omega_0 + \ddbar(u^{(k)}_{i,\delta} +
\gamma\eta_i)$ on $U_i$ satisfies $\widetilde{F}(\alpha^{p\bar m}\chi_{i,q\bar m})
\leq c -\delta'$. Note that $F(A) = S_1(A^{-1})$ here. 

Let us work at a point $z\in U_i$. Define the following four matrix
valued functions, defined in $U_i$:
\[ \begin{aligned}
  A^p_q &= \alpha^{p\bar m}\d_q\d_{\bar m} u^{(k)}_i = \alpha^{p\bar
    m} \omega_{k,q\bar m}. \\
  B^p_q &= \d_q\d_{\bar p} u^{(k)}_i \\
  C^p_q &= \d_q\d_{\bar p} u^{(k)}_{i,\delta} \\
  D^p_q &= \alpha^{p\bar m}\d_q\d_{\bar m}(u^{(k)}_{i,\delta} +
  \gamma\eta_i)
  \end{aligned}\]

Because of \eqref{eq:alphacomp}, the definition of the
mollification $u^{(k)}_{i,\delta}$ and our bound on $\gamma$, the
eigenvalues of all these matrices are at least $\kappa/2$,
and we can choose eigenvectors so that the
corresponding eigenvalues are all as close to each other as we like,
if $\tau$ is sufficiently small. In particular the same holds for the
reciprocals of the eigenvalues. In what follows, let us denote by
$h(\tau)$ a function such that $h(\tau)\to 0$ as $\tau\to 0$, and
which may change from line to line. By the structural assumption
(5) for $F_\epsilon$, we will then have
\[ F_\epsilon(B) < F_\epsilon(A) + h(\tau). \]
The convexity of $F_\epsilon$ implies that
\[ F_\epsilon(C) \leq \int_{\mathbf{C}^n} F_\epsilon(B(z - \delta
w))\rho(w)\,dw \leq \int_{\mathbf{C}^n} F_\epsilon(A(z-\delta
w))\rho(w)\,dw + h(\tau). \]
It follows that
\[ F_\epsilon(C) - c_\epsilon \leq C_\tau \Vert F_\epsilon(A) -
c_\epsilon\Vert_{L^2(\alpha)} + h(\tau), \]
where the constant $C_\tau$ blows up as $\tau\to 0$. 
Note that in our notation here $A$ is the same as $A_k$ in
Equation~\eqref{eq:limk}. Finally interchanging $C$ with $D$ will
only introduce a small error, again by structural assumption (5), so
we have
\[ F_\epsilon(D) - c_\epsilon \leq C_\tau \Vert F_\epsilon(A) -
c_\epsilon\Vert_{L^2(\alpha)} + h(\tau). \]
  
Let us write out what this means at $z$
in normal coordinates for $\alpha$, such that $\ddbar
(u^{(k)}_{i,\delta} + \gamma\eta_i)$
is diagonal with eigenvalues $\lambda_1,\ldots,
\lambda_n$ (so that these are the eigenvalues of $D$). We have
\[ \frac{1}{\lambda_1} + \ldots + \frac{1}{\lambda_n} -
\frac{\epsilon}{\lambda_1\cdots\lambda_n} -c + \epsilon \leq C_\tau \Vert
F_\epsilon(A) - c_\epsilon\Vert_{L^2} + h(\tau). \]
Since $\lambda_i > \kappa/2$, we can first choose $\epsilon$
sufficiently small so that $\epsilon(\lambda_1\cdots\lambda_n)^{-1} < 1 / \lambda_i$
for all $i$. We then choose $\tau$ so small that $h(\tau) < \epsilon /
4$, and finally, according to \eqref{eq:limk} we can choose $k$ sufficiently large, so
that $C_\tau\Vert F_\epsilon(A_k) - c_\epsilon\Vert_{L^2} <
\epsilon/4$. Combining these we have for each $i$, that
\[ \sum_{j\not=i} \frac{1}{\lambda_j} < c -
\frac{\epsilon}{2}. \]
But this means that  $\chi_i= \omega_0 +
\ddbar(u^{(k)}_{i,\delta} + \gamma\eta_i)$ satisfies
$\widetilde{F}(\alpha^{p\bar m}\chi_{i,q\bar m}) < c - \epsilon /
2$. This implies that $\chi = \omega_0 + \ddbar\psi_{k,\delta}$
satisfies $\widetilde{F}(\alpha^{i\bar p}\chi_{j\bar p}) \leq c - \epsilon / 2$ in
the viscosity sense, and so by  Theorem~\ref{thm:1visc} 
there is a metric $\omega \in
\Omega$ solving $\Lambda_\alpha\omega = c$. 
\end{proof}

\begin{remark}
  In the proof above, one could try to find a smooth metric $\chi$
  satisfying $\widetilde{F}(\alpha^{i\bar p}\chi_{j\bar p}) \leq
  c-\epsilon / 2$, by taking a regularized maximum in
  Equation~\eqref{eq:psik}. It is not at all clear, however, that the
  regularized maximum will satisfy the required subsolution
  property. This is why we take the maximum instead and work with
  subsolutions in the viscosity sense. In the proof of
  Theorem~\ref{thm:toric} we will face a similar problem, and need to
  consider  viscosity subsolutions. 
\end{remark}

The proof of Theorem~\ref{thm:2} now follows from
Proposition~\ref{prop:Jproper} below, which
is essentially contained in the work of Song-Weinkove~\cite{SW04}. 
Let us denote by $\mathcal{J}_\alpha$ the functional
$\mathcal{J}_0$ above, and by $\mathcal{J}_\beta$ the same functional
with $\beta\in [\alpha]$ replacing $\alpha$. 

\begin{prop}\label{prop:Jproper}
If there is a metric $\chi\in \Omega$ satisfying
  $\Lambda_\chi\alpha = c$, then $\mathcal{J}_\alpha$ is proper. 
In addition if $\mathcal{J}_\alpha$ is proper, then $\mathcal{J}_\beta$ is
  proper. 
\end{prop}
\begin{proof}
  Suppose that we can find an $\chi$ such that $\Lambda_\chi\alpha
  = c$. For small $\delta > 0$ the form $\alpha - \delta\chi$ is
  positive, and we have $\Lambda_\chi(\alpha - \delta\chi) = c -
  n\delta$. Song-Weinkove~\cite{SW04} showed that in this case the
  corresponding functional $\mathcal{J}_{\alpha - \delta\chi}$ is
  bounded below. This
  functional is given, up to adding a constant, by
\[ \mathcal{J}_{\alpha - \delta\chi}(\omega) =
\mathcal{J}_\alpha(\omega) - n\delta\int_0^1\int_M \phi(\chi^n -
\omega_t^n), \]
where $\omega_t = \chi + t\ddbar\phi$ and $\omega=\omega_1$. It
follows that
\[ \mathcal{J}_\alpha(\omega) \geq -C + n\delta\int_0^1 t\int_M
\sqrt{-1}\d\phi\wedge\bar{\d}\phi \wedge(\chi^{n-1} + \ldots +
\omega_t^{n-1}), \]
which implies that $\mathcal{J}_\alpha$ is proper. 

If $\beta = \alpha + \ddbar\psi$, and $\omega_t$ is as above then we have
  \[ \begin{aligned}
    \mathcal{J}_{\beta}(\omega_1) - \mathcal{J}_{\alpha}(\omega_1) &=
    \int_0^1 \int_M\phi (\beta-\alpha)\wedge
      \omega^{n-1}_t \,dt \\
  &= \int_0^1 \int_M \psi \ddbar\phi\wedge \omega_t^{n-1}\,dt \\
  &= \frac{1}{n} \int_0^1 \int_M \psi \frac{d}{dt}\omega_t^n\,dt \\
 &= \frac{1}{n}\int_M \psi(\omega_1^n - \omega_0^n). 
  \end{aligned}\] 
It follows that $|\mathcal{J}_\beta - \mathcal{J}_\alpha| < C$ for
some constant depending on $\psi$, so if $\mathcal{J}_\alpha$ is
proper, then so is $\mathcal{J}_\beta$. 
\end{proof}

The proof of Theorem~\ref{thm:indep} is very similar to the above, but
simpler. For small $\kappa, \epsilon > 0$ we consider the operator
\[ F_{\kappa,\epsilon}(A) = \sum_{k=1}^n c_k S_k(A^{-1}) + \kappa
S_1(A^{-1}) - \epsilon S_n(A^{-1}). \]
We are assuming that $c_n > 0$, so that for sufficiently small
$\epsilon$ this operator will satisfy the structural conditions on the
space of all positive Hermitian matrices. In particular for any
initial metric $\omega_0\in \Omega$, the flow
\[ \frac{\d \phi_t}{\d t} = c_{\kappa, \epsilon} -
F_{\kappa,\epsilon}(A_t) \]
has a solution for all time, with $\omega_t = \omega_0 +
\ddbar\phi_t$, $\phi_0 = 0$, and $(A_t)^i_j = \alpha^{i\bar
  k}\omega_{t,j\bar k})$. The constant $c_{\kappa,\epsilon}$ is chosen
so that
\[ \int_M c_{\kappa,\epsilon}\,\omega^n = \int_M
F_{\kappa,\epsilon}(A)\,\omega^n. \]
The flow is the negative gradient flow of the
function $\mathcal{J}_{\kappa,\epsilon}$, defined by
\[ \left.\frac{d}{dt}\right|_{t=0}
\mathcal{J}_{\kappa,\epsilon}(\omega+ t\ddbar\phi) = \int_M
\phi(F_{\kappa,\epsilon}(A) - c_{\kappa,\epsilon})\,\omega^n, \]
normalized so that $\mathcal{J}_{\kappa,\epsilon}(\omega_0) = 0$. 

The following result is not stated in this generality in
Fang-Lai-Ma~\cite{FLM11}, but it follows using exactly the same
argument (see also Song-Weinkove~\cite{SW04}), together with the
perturbation method in the proof of Proposition~\ref{prop:Jproper}. 

\begin{prop}\label{prop:Jlowerbound}
  Suppose that there is a metric $\omega$ such that
  $F_{\kappa,\epsilon} (\alpha^{i\bar
    k}\omega_{j\bar k}) = c_{\kappa,\epsilon}$. Then
  $\mathcal{J}_{\kappa,\epsilon}$ is proper. 
\end{prop}

\begin{proof}[Proof of Theorem~\ref{thm:indep}]
  We are assuming that there is a metric $\omega$ such that $F(A)=c$,
  where we are writing $F= F_{0,0}$ and $c= c_{0,0}$, and $A^i_j =
  \alpha^{i\bar k}\omega_{j\bar k}$. Using the implicit function
  theorem, we can also solve $F_{\kappa,\epsilon}(A) = c_{\kappa,
    \epsilon}$, for sufficiently small $\kappa,\epsilon > 0$, and this
  implies that $\mathcal{J}_{\kappa,\epsilon}$ is proper. Let us write
  $\mathcal{J}'_{\kappa,\epsilon}$ for the functional defined in the
  same way as $\mathcal{J}_{\kappa,\epsilon}$ but with $\alpha$
  replaced by a metric $\beta\in[\alpha]$. Just as in
  Proposition~\ref{prop:Jproper} we obtain that
  $\mathcal{J}'_{\kappa,\epsilon}$ is also proper. We can then use the
  negative gradient flow to obtain a sequence of metrics $\omega_k$,
  such that the matrices $B_k$ defined by $(B_k)^i_j = \beta^{i\bar
    p}\omega_{k,j\bar p}$ satisfy
  \[ \lim_{k\to\infty} \int_M (F_{\kappa,\epsilon}(B_k) - c_{\kappa,\epsilon})^2\,
  \omega_k^n = 0. \]
  In addition we have $F_{\kappa, \epsilon}(B_k) < C$ along the flow,
  for a uniform constant $C$, and so using that $F_{\kappa, \epsilon}(B_k)
  > \kappa S_1(B_k^{-1})$, we obtain a uniform lower bound $\omega_k > C^{-1}
  \beta$, where $C$ also depends on $\kappa$.

   Performing the same mollification argument as in the proof of
   Proposition~\ref{prop:properconverge}, for any $\delta > 0$ we can
   obtain a K\"ahler current $\chi\in[\omega]$ 
  with continuous local potentials, satisfying
  \[ F_{\kappa,\epsilon}(\beta^{i\bar p}\chi_{j\bar p}) \leq
  c_{\kappa,\epsilon} + \delta \]
  in the viscosity sense. We have
  \[ \int_M c_{\kappa,\epsilon}\omega^n = \int_M
  F_{\kappa,\epsilon},\omega^n = \int_M c\,\omega^n + \int_M \Big[\kappa
  S_1(A^{-1}) - \epsilon S_n(A^{-1})\Big]\,\omega^n, \]
  and so $c_{\kappa,\epsilon} = c + \kappa d_1 - \epsilon d_2$ for
  some positive constants $d_1,d_2 > 0$, so 
  \[ F_{\kappa,\epsilon}(\beta^{i\bar p}\chi_{j\bar p}) \leq
  c + \delta + \kappa d_1 - \epsilon d_2. \]
  Choosing $\kappa$ sufficiently small so that $\kappa d_1 - \epsilon
  d_2 < 0$, and then $\delta$ sufficiently small, we will have
  \[ F_{\kappa,\epsilon}(\beta^{i\bar p}\chi_{j\bar p}) \leq c -
  \delta', \]
  for some small $\delta' > 0$. 
  The definition of $F_{\kappa,\epsilon}$ then implies
  \[ \sum_{k=1}^{n-1} c_k S_k(B^{-1}) \leq c - \delta' \]
  in the viscosity sense, where $B^i_j = \beta^{i\bar p}\omega_{j\bar
    p}$. Theorem~\ref{thm:1visc} then implies that we can find a
  metric $\eta\in [\omega]$ satisfying $F(\beta^{i\bar p}\eta_{j\bar
    p}) = c$.  
\end{proof}

\section{Toric manifolds}\label{sec:toric}

In the remainder of this article we will work on a toric manifold
$M$. Our goal is to prove Theorem~\ref{thm:3}. By Theorems~\ref{thm:2}
and \ref{thm:indep}
it is sufficient to work with torus invariant metrics. We restate the
theorem here. 
\begin{thm}\label{thm:toric}
Let $M$ be a toric manifold, and $\alpha, \chi$ torus
invariant K\"ahler metrics on $M$.
  Suppose that $c > 0$ is such that
  \begin{equation}\label{eq:Mc}
\int_M c\chi^n - n\chi^{n-1}\wedge\alpha \geq 0, \end{equation}
  and for all toric subvarieties $V\subset M$ of dimension
  $p=1,2,\ldots, n-1$ we have
  \[ \int_V c\chi^p - p\chi^{p-1}\wedge\alpha > 0. \]
  Then there exists a torus invariant metric $\omega\in [\chi]$ such
  that
  \begin{equation}\label{eq:twisted} \Lambda_\omega\alpha +
    d\frac{\alpha^n}{\omega^n}= c 
    \end{equation}
  for some constant $d\geq 0$, depending on the choice of $c$. In particular either
  $\Lambda_\omega\alpha = c$ i.e. we have a solution of the
  J-equation,  or $\Lambda_\omega\alpha < c$, depending on whether we
  have equality in \eqref{eq:Mc}. 
\end{thm}
\begin{proof}
  The proof proceeds by induction on the dimension of $M$, the result
  being straight forward when $\dim M = 1$. Let us assume that we
  already know the result for dimensions less than $n$, and suppose that $\dim M =
  n$. We solve Equation~\eqref{eq:twisted}, by the continuity method just as
  in the proof of Theorem~\ref{thm:1}. 

  Using Yau's theorem, we can find $\omega\in[\chi]$ such that
  $\frac{\alpha^n}{\omega^n}$ is constant, which corresponds to the
  limit $c\to\infty$. From the implicit function theorem it then
  follows that there is some $c'$, such that we can solve
  Equation~\eqref{eq:twisted} for all $c\in (c', \infty)$ with $d$
  depending on $c$. Define $c_0$ to
  be the infimum of all such $c'>0$. Our goal is to show that if $c_0
  \geq c$, and $c_k \to c_0$, then we have uniform estimates
  $C^{k,\alpha}$ estimates for the solutions $\omega_k$ of the
  equations
  \[ \Lambda_{\omega_k}\alpha + d_k \frac{\alpha^n}{\omega^n} = c_k. \] 
  Let us denote by $D = \cup_{i=1}^N D_i$ the union of all toric divisors on
  $M$.

  We use the inductive hypothesis to build a suitable subsolution for
  the equation in a neighborhood of $D$.
  For each $D_i$, the inductive hypothesis implies that there is some
  $d_i > 0$, and a form $\omega_i\in [\chi]$, such that the
  restriction of
  $\omega_i$ to $D_i$ is positive, and satisfies
  \[ \Lambda_{\alpha|_{D_i}}\omega_i|_{D_i} + d_i
  \frac{\alpha|_{D_i}^{n-1}}{\omega_i|_{D_i}^{n-1}} = c. \] 
 We define
  \[ \chi_i = \omega_i + A\ddbar(\gamma(d_i) |d_i|^2), \]
  where $d_i$ denotes the distance from $D_i$, $A$ is a large
  constant, and $\gamma:\mathbf{R}\to\mathbf{R}$ is a cutoff function
  supported near 0. If $A$ is chosen sufficiently large, then $\chi_i$
  will be positive in a small neighborhood $U_i$ of $D_i$ and will
  satisfy  $\Lambda_\alpha \chi_i < c - \kappa$ for some small $\kappa
  > 0$. 

  Fixing a reference form $\omega_0$, we can write $\chi_i = \omega_0
  + \ddbar\psi_i$ for each $i$. Consider now the functions
  \[ \widetilde{\psi}_i = \psi_i -B_i +  \delta\sum_{j< i}
  \gamma(d_j)\log d_j, \]
  where $B_i, \delta > 0$ are constants, and let $\widetilde{\chi}_i = \omega_0 +
  \ddbar\widetilde{\psi}_i$. 
  We can choose $\delta$ sufficiently small, so that on a neighborhood
  of $D_i\setminus
  \cup_{j < i}D_j$, the form $\widetilde{\chi}_i$ is positive
  definite and satisfies $\Lambda_{\widetilde{\chi}_i}\alpha <
  c-\kappa$.  We choose $B_i$ inductively, for $i=N, N-1, \ldots, 1$,
  starting with $B_N=0$, so that on $D_j$, for all $j > i$ we have
  $\psi_i - B_i < \psi_j - B_j$. 

  Suppose that $x\in M$ is in a neighborhood $V$ of an intersection
  $D_{i_1}\cap \ldots \cap D_{i_k}$, where $i_1 <\ldots < i_k$, and
  suppose that
  \[ \widetilde{\psi}_a(x) = \max_j \widetilde{\psi}_j(x). \]
  If $V$ is sufficiently small, then we must have $a\leq i_1$, since
  $\widetilde{\psi}_j = -\infty$ along $D_i$ for $i < j$. Together
  with the choice of the $B_i$, this implies that we must have
  $a=i_1$, if the neighborhood $V$ is sufficiently small. We can also
  assume that $V$ is disjoint from $D_j$ for $j < i_1$, so that we
  have $\Lambda_{\widetilde{\chi}_{i_1}}\alpha < c-\kappa$ on
  $V$. Using Lemma~\ref{lem:maxviscosity} 
   we have that on a sufficiently small neighborhood $U$ of $D =
  \bigcup_i D_i$ the K\"ahler current $\chi = \omega_0 + \max
  \widetilde{\psi}_i$ satisfies $\widetilde{F}(\alpha^{i\bar
    p}\chi_{j\bar p}) \leq c-\kappa$ in the viscosity sense, where
  $F(A) = S_1(A^{-1})$. We can therefore apply 
  Proposition~\ref{prop:BenC2}, reducing $C^2$-esimates on $U$ to the
  boundary $\d U$. 

  In Proposition~\ref{prop:interior} below, we will show that we have
  uniform $C^{l,\alpha}$ estimates for the $\omega_k$ outside the
  neighborhood $U$ of $D$, and so we can apply Proposition~\ref{prop:BenC2} to the
      closure of $U$ to obtain bounds of the form
      \[ \Lambda_\alpha \omega_k < Ce^{N(\phi_k - \inf\phi_k)} \]
      on $U$, where $\omega_k = \omega_0 + \ddbar\phi_k$. Since we already have
      estimates outside of $U$, the same inequality is true globally on
      $M$. Just as in the proof of Theorem~\ref{thm:1}, we can use
      Weinkove~\cite[Proposition 4.2]{Wei04} to obtain the estimate
      $\Lambda_\alpha\omega_k < C$, and as in the proof of
      Theorem~\ref{thm:1visc}
      the Evans-Krylov theorem and Schauder estimates can
      be used to obtain higher order estimates. 
\end{proof}

\subsection{Interior estimates}
In this section we work on a toric manifold $M$ with torus invariant
K\"ahler metrics $\alpha, \omega$, satisfying the equation
\begin{equation}\label{eq:toriceq}
  S_1(A) + dS_n(A) = c, \end{equation} 
where $A^i_j = \omega^{i\bar p}\alpha_{j\bar p}$, and $d, c$ are
non-negative constants. 
Our goal is to obtain estimates for
$\omega$ in terms of $\alpha, c, d$, away from the torus invariant
divisors of $M$. 

\begin{prop}\label{prop:interior}
  Suppose that $\alpha,\omega$ are torus invariant metrics on $M$
  satisfying Equation~\eqref{eq:toriceq}.
  Then on any compact set
  $K\subset M$ disjoint from the torus invariant divisors, we have
  bounds
  \[ \begin{aligned} \omega &> C^{-1}\alpha, \\
    \Vert\omega\Vert_{C^{2,\alpha}} &< C,
  \end{aligned}\]
  for $C$ depending on $M, \alpha$, bounds on $c,d$, and the K\"ahler class $[\omega]$. 
\end{prop}

The proof of this proposition will occupy the rest of this section. 
To obtain these estimates, we write our equation in terms of convex
functions on $\mathbf{R}^n$, corresponding to the dense complex torus
$\mathbf{R}^n\times (S^1)^n$
in $M$. Suppose that $\alpha = \ddbar f$ and $\omega = \ddbar g$,
where $f, g: \mathbf{R}^n\to \mathbf{R}$ are convex. We are assuming
that $f, g$ satisfy the equation
\[ S_1(A) + dS_n(A) = c, \]
where $A^i_j = g^{ip}f_{jp}$. The function $f$ is fixed, and we want to derive
estimates for the function $g$ on compact sets $K\subset
\mathbf{R}^n$. By adding an affine linear function to $g$, we can assume that
$g(0)=0$ and $\nabla g(0)=0$.

\begin{lemma} For any compact $K\subset \mathbf{R}^n$
  there exists a $C > 0$ such that $\sup_K|g| < C$.
\end{lemma}
\begin{proof}
  The image of $\nabla g$ is a convex polytope $P$, determined by the
  K\"ahler class of $\omega$ up to
  translation by adding affine linear functions to $g$. Our
  normalization ensures that $0\in P$, and in particular we get a
  gradient bound on $g$. The result follows immediately. 
\end{proof}

We next prove a $C^2$-estimate for $g$ on compact sets by a
contradiction argument.
\begin{prop} Suppose that $f,g : B\to\mathbf{R}$ are convex functions
  on the unit ball, satisfying
  \[ f_{ij}g^{ij} + d \frac{\det(D^2 f)}{\det(D^2 g)} = c, \]
  with $d\geq 0$ as above, and suppose $\inf_B g=g(0)=0$.
  Then there is a $C> 0$ depending on $\sup_B |g|$, bounds on $c,d$, 
  $C^{3,\alpha}$ bounds on $f$ and a lower bound on the Hessian of $f$, such that
  \[ \sup_{\frac{1}{2}B} |g_{ij}| < C. \]
\end{prop}
\begin{proof}
  We can assume that $c=1$ by scaling $f$. 
  We argue by contradiction. Suppose that we have sequences $f_k, g_k$
  satisfying the hypotheses, including $|g_k| < N$, but $|\partial^2 g_k(x_k)| > k$ for some
  $x_k\in \frac{1}{2}B$. Note that the equation implies that
  \[ g_{k,ij} > f_{k, ij} > \tau\delta_{ij}, \]
  for some fixed $\tau > 0$, i.e. we have a uniform lower bound on the Hessians of the $g_k$.

  Let $h_k : U_k \to \mathbf{R}$ be the Legendre transform of
  $g_k$. By shrinking the ball a bit, we can assume that $g_k \to g$
  uniformly for some strictly convex $g: B\to \mathbf{R}$. Lemma~\ref{lem:image}
  below implies that for sufficiently large $k$ we have $U_k \supset
  \nabla g(0.9B)$, and so $\nabla g(0.8B)$ is of a definite distance
  from $\partial U_k$ for large $k$. In addition, $h_k$
  satisfies the equation 
  \begin{equation}\label{eq:Legtransf}
    \sum_{i,j} f_{k, ij}(\nabla h_k(y)) h_{k, ij}(y) + d_k \det(D^2 f_k(\nabla h_k))
    \det(D^2 h_k(y))= 1.
  \end{equation}
  In addition from the
  normalization we get
  $h_k(0) = \nabla h_k(0) =0$. We use Proposition~\ref{prop:C2alpha} below together with
  the Schauder estimates, to obtain
  uniform $C^{l,\alpha}$ bounds on each $h_k$, on $\nabla f(0.8B)$, so 
  we can take a limit $h_\infty : \nabla f(0.8B) \to \mathbf{R}$, satisfying an equation
  of the form
  \[ \sum_{ij} f_{\infty, ij}(\nabla h_\infty(y))h_{\infty, ij}(y) +
  d_\infty \det(D^2f_\infty(\nabla h_\infty(y)))  \det(D^2 h_\infty(y))
      =
  1. \]
  There are two cases:
  \begin{enumerate}
  \item We have a positive lower bound on the Hessian of
    $h_{\infty}$. This implies a lower bound on $\mathrm{Hess}\, h_k$
    for large $k$, on $\nabla f(0.8B)$, i.e. we get an upper bound on
    $\mathrm{Hess}\, g_k$ at all points $x\in B$ for which $\nabla
    g_k(x)\in \nabla f(0.8B)$. But by  Lemma~\ref{lem:image}, $\nabla g_k(0.7B)
    \subset \nabla g(0.8B)$ for large $k$, so we get an upper bound on
    $\mathrm{Hess}\,g_k$ on $0.7B$, which contradicts our assumption.
  \item The Hessian of $h_\infty$ is degenerate somewhere. Then we can apply
    the constant rank theorem of Bian-Guan~\cite[Theorem
    1.1]{BG09}. Indeed, for a fixed value of $\nabla h_\infty$, the
    equation is of the form
    \[ F(A) = \mathrm{Tr}(BA) + c \det(A) - 1 = 0, \]
    for a positive definite matrix $B$ and positive constant $c$. The
    assumptions of Theorem 1.1 in \cite{BG09} are satisfied, since the
    map $A\mapsto F(A^{-1})$ is convex in $A$, according to
    Lemma~\ref{lem:convex}. It follows that if the Hessian of
    $h_\infty$ is degenerate at a point, then it
    must be degenerate everywhere, and so 
    \[ \int_{ \nabla g(0.8B)} \det(h_{\infty,ij}) = 0. \]
    This contradicts the fact that $\nabla g_k(0.7B) \subset \nabla
    g(0.8B)$ for large $k$, and
\[  \int_{ \nabla g_k(0.7B) } \det(h_{k, ij}) = \mathrm{Vol}(0.7B), \]
  but $h_k\to h_\infty$ in $C^{2,\alpha}$. 
  \end{enumerate}
\end{proof}

\begin{lemma}\label{lem:image}
 Suppose that $f_k:B\to\mathbf{R}$ are convex, with $f_{k,ij} >
 \tau\delta_{ij}$, such that they converge uniformly to $f :
 B\to\mathbf{R}$. If $B_1 \subset B_2 \subset B_3\subset B$ are relatively compact
 balls, then for sufficiently large $k$ 
 the gradient maps satisfy
\[ \nabla f_k(B_1) \subset \nabla f(B_2) \subset \nabla f_k(B_3).\] 
\end{lemma}
\begin{proof}
  From Guti\'errez~\cite[Lemma 1.2.2]{Gut01} we have that 
 \[\limsup_{k\to\infty} \nabla f_k(K)
  \subset \nabla f(K),\]
  for any compact $K\subset B$. 
  Suppose that $B_1 \subset B' \subset B_2$. 
  The strict convexity of $f$ implies that $\nabla f(\partial B_2)$ is
  a positive distance from $\nabla
  f(\overline{B'})$. In particular for each $x\in \nabla f(\partial
  B_2)$, there is a $k_x$, such that $x\not\in \nabla
  f_k(\overline{B'})$ for all $k > k_x$. The strict convexity then
  implies that there is a (fixed) radius $\delta > 0$, such that
  $B_\delta(x)$ is disjoint from $\nabla f_k(\overline{B_1})$ for all
  $k > k_x$. Since $\nabla f(\partial B_2)$ is compact by \cite[Lemma
  1.1.3]{Gut01}, 
  we can find
  some $N$ such that $\nabla f_k(\overline{B_1})$ is disjoint from
  $\nabla f(\partial B_2)$ for all $k >N$. But this implies
\[ \nabla f_k(B_1) \subset \nabla f(B_2) \]
  for $k>N$. 

  For the other inclusion we use that for any compact $K\subset B$ and
  open set $U\supset K$ with $\overline{U}\subset B$, we have
  \[ \nabla f(K) \subset \liminf_{k\to\infty} \nabla f_k(U). \]
  Now choose an intermediate ball $B'$ with $B_2\subset B' \subset
  B_3$. For any $x\in \nabla f(\overline{B}_2)$, we have a $k_x$ such
  that $x\in \nabla f_k(B')$ for all $k > k_x$. By the strict
  convexity we have some $\delta > 0$ such that $B_\delta(x) \in
  \nabla f_k(B_3)$. Using that $\nabla f(\overline{B}_2)$ is compact,
  we can again cover by finitely many such balls, and we get an $N$
  such that $\nabla f(\overline{B}_2) \subset \nabla f_k(B_3)$ for all
  $k > N$.  
\end{proof}

The higher order estimates required by Proposition~\ref{prop:interior} 
follow from standard elliptic theory. We now show the
$C^{2,\alpha}$-estimates for equation~\eqref{eq:Legtransf}. The difficulty
is that the operator is neither concave, nor convex, and the result of
Caffarelli-Yuan~\cite{CY00} also does not apply directly. 

\begin{prop}\label{prop:C2alpha}
  Suppose that $h : B\to\mathbf{R}$ is a smooth convex function on the
  unit ball in $\mathbf{R}^n$ satisfying the equation
  \begin{equation}\label{eq:heqn} \sum_{i,j} a_{ij}(\nabla h)h_{ij} +
    b(\nabla h)
    \det(D^2h) = 1, 
\end{equation}
  where $a_{ij}, b \in C^{1,\alpha}$ and $\lambda < a_{ij} <
  \Lambda$. Then we have $\Vert h\Vert_{C^{2,\alpha}(\frac{1}{2}B)} <
  C$ for a constant $C$ depending on $\lambda, \Lambda$ and
  $C^{1,\alpha}$ bounds for $a_{ij}, b$ and $C^1$ bounds for $h$. 
\end{prop}
As a first step we prove a priori $C^{2,\alpha}$ estimates for the constant coefficient equation.

\begin{prop}\label{prop:C2alpha const}
  Suppose that $h : B\to\mathbf{R}$ is a smooth convex function on the
  unit ball in $\mathbf{R}^n$ satisfying the equation
  \begin{equation}\label{eq:modelheqn} 
  \Delta h + b\det(D^2h) = 1, 
\end{equation}
  where $b \geq0$ is a non-negative constant. Then we have $\Vert h\Vert_{C^{2,\alpha}(\frac{1}{2}B)} <
  C$ for a constant $C$ depending on $b$ and $|h|_{C^{1,\alpha}(B)}$. 
  \end{prop}
 \begin{proof}
We will assume $b>0$, since $b=0$ is standard. Let $f = \det(D^2h)^{1/n}$, and denote
the linearized operator by $L$, which acts on a smooth function
$g$ by
\[ L g= \Delta g + b \det(D^2h) h^{ij} g_{ij}, \]
where we use summation convention for repeated indices.  We now
compute $Lf$. We work at a point where $D^2h$ is diagonal. We also
write $S_n = \det(D^2 h)$ to simplify notation. 
 First, differentiating the equation we have for each $k$
\[ \begin{gathered}
          h_{iik} + bS_n h^{ij}h_{ijk} = 0, \\
          h_{iikk} (1 + bS_nh^{ii})+ bS_n h^{pp}h_{ppk}
          h^{qq}h_{qqk} - bS_n h^{pp}h^{qq} h_{pqk}^2 = 0.
\end{gathered} \]
Also, differentiating $f$, we have
\[ \begin{gathered}
 f_k = \frac{1}{n} S_n^{1/n} h^{ij}h_{ijk} \\
 f_{kk} = \frac{1}{n} S_n^{1/n} h^{ii}h_{iikk} + \frac{1}{n^2} S_n^{1/n} h^{pp}h_{ppk} h^{qq}h_{qqk}
             - \frac{1}{n}S_n^{1/n} h^{pp}h^{qq} h_{pqk}^2. 
\end{gathered}\]
We now compute $Lf$:
\[ 
\begin{aligned}
 Lf &= (1 + bS_nh^{kk}) f_{kk} \\
&= \frac{S_n^{1/n}}{n^2}(1 + bS_nh^{kk})\Big[nh^{ii}h_{iikk} +
h^{pp}h^{qq}(h_{ppk}h_{qqk} - nh_{pqk}^2)\Big] \\
&= \frac{S_n^{1/n}}{n}\Big[ h^{ii}bS_nh^{pp}h^{qq}(h_{pqi}^2 -
h_{ppi}h_{qqi})\Big] \\
&\quad + \frac{S_n^{1/n}}{n^2}(1 + bS_nh^{ii})
h^{pp}h^{qq}(h_{ppi}h_{qqi} - nh_{pqi}^2) \\
&= \frac{S_n^{1/n}}{n^2}h^{ii}h^{pp}h^{qq}h_{ppi}h_{qqi}(-nbS_n
+ bS_n) \\
&\quad + \frac{S_n^{1/n}}{n^2} \sum_{i,p,q}
h^{pp}h^{qq}h_{ppi}h_{qqi}- \frac{S_n^{1/n}}{n} \sum_{i,p,q} 
h^{pp}h^{qq}h_{pqi}^2. 
\end{aligned}\]
We have
\[ \begin{aligned}
\sum_{p,q} h^{pp}h^{qq}h_{ppi}h_{qqi} = \left(\sum_p
  h^{pp}h_{ppi}\right)^2 \leq n\sum_p (h^{pp}h_{ppi})^2 
&\leq n\sum_{p,q} h^{pp}h^{qq}h_{pqi}^2. 
\end{aligned}
\]
It follows from this that $Lf \leq 0$. 

At this point we can follow the argument in
Caffarelli-Yuan~\cite{CY00} closely.  The only difference in the
argument is that in \cite{CY00} the function $e^{K\Delta u}$ is a
subsolution of the Linearized equation for a sufficiently large constant
$K$. This does not appear to be the case for our equation, but instead
we can use that  $\det(D^2u)^{1/n}$ is a supersolution according to
our calculation above. In \cite{CY00} this supersolution property is
only used in dealing with ``Case 2'' in the proof of their Proposition
1. We will see that the same argument works in our situation as well.

As in \cite{CY00}, we fix $\rho, \xi, \delta, k_{0}> 0$ to  
be determined and we set  
\[s_{k} := \sup_{x \in B_{1/2^{k}}} \Delta u(x), \qquad 1 \leq k \leq k_{0}. 
\]
From the equation we know that $s_{k} \leq 1$. Define
\[
E_{k} := \left\{x\in B_{1/2^{k}} |\, \Delta u(x) \leq s_k-\xi \right\},
\]
and as in Case 2 in \cite{CY00}, assume that for all $1\leq k\leq k_0$
we have $|E_k| > \delta |B_{1/2^k}|$. 
Let us define
\[w_{k}(x) = 2^{2k} u\left(\frac{x}{2^{k}}\right).
\]
We apply the above computation to conclude that
\[
L\left[ (1-\Delta w_k)^{1/n} - (1-s_k)^{1/n}\right] \leq 0
\]
Since $ (1-\Delta w_k)^{1/n} - (1-s_k)^{1/n} \geq 0$ on $B_1$,
we can apply the weak Harnack inequality 
to obtain
\[
(1-s_{k+1})^{1/n} = \inf_{B_{1/2}}(1-\Delta w_k)^{1/n}  \geq (1-s_{k})^{1/n}+ c \|
(1-\Delta w_k)^{1/n} - (1-s_k)^{1/n} \|_{L^{p_0}(B_{1})},
\]
on $B_{1/2}$, for uniform constants $c, p_0 > 0$. 
The right hand side can be estimated using the assumption on the
measure $|E_{k}|$, and we get
\[
\inf_{B_{1/2}} (1-\Delta w_k)^{1/n} \geq (1-s_k)^{1/n} + c
\delta^{1/p_{0}} \frac{\xi}{n},
\]
since we have
\[ (1-s_k + \xi)^{1/n} - (1-s_k)^{1/n} \geq \frac{\xi}{n}. \]
Since $1-s_k \leq 1$, there can only be at most a bounded number $k_0$
of such steps. The remainder of the argument is identical to
Cafarelli-Yuan \cite{CY00}. 
\end{proof}

We can now prove the interior $C^{2,\alpha}$ estimates for the general equation in
Proposition~\ref{prop:C2alpha} by using a blow-up argument.  We begin by proving
a Liouville rigidity theorem for convex solutions of our equation;

\begin{lemma}\label{lem:Liouville}
Suppose that $u: \mathbf{R}^{n} \rightarrow \mathbf{R}$ is a smooth, convex function
satisfying
\begin{equation}\label{eq: twist eq on Rn}
\Delta u + b\det(D^2u) =1,
\end{equation}
then $u$ is a quadratic polynomial.
\end{lemma}
\begin{proof}
Th lemma follows from a simple
rescaling argument, which is essentially the same as the proof of
Guti\'errez~\cite[Theorem 4.3.1]{Gut01}.
  By subtracting a plane we may assume that $u(0)= 0$ and
$\nabla u(0)=0$.  Since $u$ is convex, the equation implies that $|D^2u| \leq \sqrt{n}$.
By integration we obtain the bound $|\nabla u|(x) \leq \sqrt{n} |x|$.  We consider the
 function $v_{R}(x) := R^{-2}u(Rx)$.  By the above, for $x \in B_{1}$ there holds
 \[ v_{R}(0)=0, \qquad |\nabla v_{R}|(x) \leq \sqrt{n}, \qquad D^{2}v_{R}(x) = D^{2}u(Rx), \]
and so $v_{R}(x)$ is uniformly bounded in $C^2(B_{1})$.
 Moreover, $v_{R}(x)$ solves equation~(\ref{eq: twist eq on Rn}) on $B_{1}$, and so  
 by the interior estimates in Proposition~\ref{prop:C2alpha const}
 we have a uniform bound for $|D^2v_{R}|_{C^{\alpha}(B_{1/2})}$.  Writing this in terms of $u$, we have
 \[
 R^{\alpha} |D^2u|_{C^{\alpha}(B_{2^{-1}R})} \leq C
 \]
 where $C$ is independent of $R$.  Taking the limit as $R\rightarrow \infty$ we see that
 we must have $D^2u = D^2u(0)$ a constant.  Hence $u$ is a quadratic polynomial.
 \end{proof}

\begin{proof}[Proof of Proposition~\ref{prop:C2alpha}]

To deal with the case of varying coefficients, we use a blowup
argument to reduce to the Liouville result in the constant coefficient
case.

Suppose then that $h$  satisfies equation~\eqref{eq:heqn} on $B$. Let
\[ N_h = \sup_{x\in B} d_x|D^3h(x)|, \]
where $d_x = d(x, \d B)$ is the distance to the boundary of $B$. Our
goal is to bound $N_h$ from above, so we can assume $N_h > 1$ say. 
Let us assume that the supremum is achieved at a point $x=x_0\in B$.  We define the function 
\[ \widetilde{h}(z) = d_{x_{0}}^{-2}N_h^2 h(x_0 + d_{x_{0}}N_h^{-1} z) - A - A_iz_i,\]
where $A, A_i$ are constants chosen so that 
\begin{equation}\label{eq: blow-up normal}
\widetilde{h}(0) =0 ,\qquad
\nabla\widetilde{h}(0) = 0.
\end{equation}
The function $\widetilde{h}(z)$ is
defined on the ball $B_{N_h}(0)$ around the origin. By direct computation we have
\[
D^{2}\widetilde{h}(z) = D^{2}h(x_{0} + d_{x_{0}}N_{h}^{-1}z), \qquad D^{3}\widetilde{h}(z) = d_{x_{0}}N_{h}^{-1}D^3h(x_{0} + d_{x_{0}}N_{h}^{-1}z)
\]
and so $|D^3\widetilde{h}(z)| \leq 2$ on $B_{2^{-1}N_{h}}(0)$.
Moreover, since $|D^2\tilde h| = |D^{2}h| <C$, the normalization~(\ref{eq: blow-up normal})
implies that we have a bound
\[ \Vert \widetilde{h} \Vert_{C^3(B_{2^{-1} N_h })} < C \]
for a uniform constant $C$.   
In addition $\widetilde{h}$ satisfies an equation of the form
\begin{equation}\label{eq:htildeeq}
 \sum_{i,j} \widetilde{a}_{ij}(\nabla \widetilde{h})
\widetilde{h}_{ij} + \widetilde{b}(\nabla\widetilde{h}) \det(D^2
\widetilde{h}) = 1, 
\end{equation}
for coefficients $\widetilde{a}_{ij}, \widetilde{b}$ satisfying the
same bounds as $a_{ij}, b$, but 
\begin{equation}\label{eq:scaleab}
\begin{aligned} \sup |\nabla \widetilde{a}_{ij}| &\leq d_{x_{0}}N_h^{-1}\sup
  |\nabla a_{ij}|, \\
\sup |\nabla \widetilde{b}| &\leq d_{x_{0}}N_h^{-1} \sup |\nabla b|.
\end{aligned}
\end{equation}
Differentiating equation~\eqref{eq:htildeeq}, we obtain uniform
$C^{2,\alpha}$ bounds on $\nabla\widetilde{h}$ on compact subsets of
$B_{N_h / 4}$.

For the sake of obtaining a contradiction we suppose 
that we have a sequence of convex functions $h_k$ on $B$ satisfying
\eqref{eq:heqn}, such that the corresponding constants
$N_{h_k} >4k$. Then the rescaled functions $\widetilde{h}_k$ are
 defined on $B_{4k}(0)$ and have uniform
$C^{3,\alpha}$ bounds on $B_{k}(0)$, and satisfy
$|D^3\widetilde{h}_k(0)|=1$.  By taking a diagonal subsequence, we can extract
a convex limit $\widetilde{h}_\infty : \mathbf{R}^n
\to\mathbf{R}$ in $C^{3,\alpha/2}$, satisfying
$|D^3\widetilde{h}_\infty(0)| = 1$, and equation~\eqref{eq:htildeeq}
with constant coefficients because of \eqref{eq:scaleab}.  Since $\widetilde{h}_{\infty}$ is convex,
 $C^{3,\alpha/2}$ on $\mathbf{R}^{n}$ and satisfies the constant
 coefficient equation 
  \eqref{eq: twist eq on Rn} (after a linear change of coordinates),
  we easily obtain that  
 $\widetilde{h}$ is in fact smooth.  In particular, we can apply the
 Liouville rigidity result in Lemma~\ref{lem:Liouville} to 
 conclude that $\widetilde{h}_{\infty}$ is a quadratic polynomial.
 But this contradicts $|D^3\widetilde{h}_\infty(0)| = 1$. 
\end{proof}

\subsection*{Acknowledgements}
The authors would like to thank Connor Mooney and Ovidiu Savin for
helpful suggestions.  The authors are also grateful to D.H. Phong and S.-T. Yau for
their encouragement and support. The second named author is supported by NSF
grants DMS-1306298 and DMS-1350696.


\begin{thebibliography}{10}

\bibitem{BG09}
{\sc B.~Bian and P.~Guan}, {\em A microscopic convexity principle for nonlinear
  partial differential equations}, Invent. Math., 177 (2009), pp.~307--335.

\bibitem{BK07}
{\sc Z.~B\l{}ocki and S.~Ko\l{}odziej}, {\em On regularization of
  plurisubharmonic functions on manifolds}, Proc. Amer. Math. Soc., 135 (2007),
  pp.~2089--2093.

\bibitem{CC95}
{\sc L.~Caffarelli and X.~Cabr\'e}, {\em Fully nonlinear elliptic equations},
  vol.~43 of American Mathematical Society Colloquium Publications, American
  Mathematical Society, Providence, RI, 1995.

\bibitem{CNS3}
{\sc L.~Caffarelli, L.~Nirenberg, and J.~Spruck}, {\em The {D}irichlet problem
  for nonlinear second-order elliptic equations {III}: {F}unctions of the
  eigenvalues of the {H}essian}, Acta Math., 155 (1985), pp.~261--301.

\bibitem{CY00}
{\sc L.~Caffarelli and Y.~Yuan}, {\em A priori estimates for solutions of fully
  nonlinear equations with convex level set}, Indiana Univ. Math. J., 49
  (2000), pp.~681--695.

\bibitem{Chen00}
{\sc X.~X. Chen}, {\em On the lower bound of the {M}abuchi energy and its
  application}, Int. Math. Res. Notices, 12 (2000), pp.~607--623.

\bibitem{Chen04}
\leavevmode\vrule height 2pt depth -1.6pt width 23pt, {\em A new parabolic flow
  in {K}\"ahler manifolds}, Comm. Anal. Geom., 12 (2004), pp.~837--852.

\bibitem{DP04}
{\sc J.-P. Demailly and M.~Paun}, {\em Numerical characterization of the
  {K}\"ahler cone of a compact {K}\"ahler manifold}, Ann. of Math. (2), 159
  (2004), pp.~1247--1274.

\bibitem{Don99}
{\sc S.~K. Donaldson}, {\em Moment maps and diffeomorphisms}, Asian J. Math., 3
  (1999), pp.~1--16.

\bibitem{Don02}
\leavevmode\vrule height 2pt depth -1.6pt width 23pt, {\em Scalar curvature and
  stability of toric varieties}, J. Differential Geom., 62 (2002),
  pp.~289--349.

\bibitem{Ev82}
{\sc L.~C. Evans}, {\em Classical solutions of fully nonlinear, convex, second
  order elliptic equations}, Comm. Pure Appl. Math., 25 (1982), pp.~333--363.

\bibitem{FL12}
{\sc H.~Fang and M.~Lai}, {\em Convergence of general inverse $\sigma_k$-flow
  on {K}\"ahler manifolds with {C}alabi ansatz}, arXiv:1203.5253.

\bibitem{FLM11}
{\sc H.~Fang, M.~Lai, and X.~Ma}, {\em On a class of fully nonlinear flows in
  {K}\"ahler geometry}, J. Reine Angew. Math., 653 (2011), pp.~189--220.

\bibitem{GS13}
{\sc B.~Guan and W.~Sun}, {\em On a class of fully nonlinear elliptic equations
  on {H}ermitian manifolds}, arXiv:1301.5863.

\bibitem{GZ07}
{\sc V.~Guedj and A.~Zeriahi}, {\em The weighted {M}onge-{A}mp\`ere energy of
  quasiplurisubharmonic functions}, J. Funct. Anal., 250 (2007), pp.~442--482.

\bibitem{Gut01}
{\sc C.~E. Guti\'errez}, {\em The {M}onge-{A}mp\`ere equation}, Progress in
  {N}onlinear {D}ifferential {E}quations and their {A}pplications, 44,
  Birkh\"auser Boston Inc., 2001.

\bibitem{Kry82}
{\sc N.~V. Krylov}, {\em Boundedly nonhomogeneous elliptic and parabolic
  equations}, Izvestia Akad. Nauk. SSSR, 46 (1982), pp.~487--523.

\bibitem{LSz13}
{\sc M.~Lejmi and G.~Sz\'ekelyhidi}, {\em The {J}-flow and stability},
  arXiv:1309.2821.

\bibitem{SW04}
{\sc J.~Song and B.~Weinkove}, {\em On the convergence and singularities of the
  {J}-flow with applications to the {M}abuchi energy}, Comm. Pure Appl. Math.,
  61 (2008), pp.~210--229.

\bibitem{Spruck05}
{\sc J.~Spruck}, {\em Geometric aspects of the theory of fully nonlinear
  elliptic equations}, in Global theory of minimal surfaces, vol.~2, Amer.
  Math. Soc., Providence, RI, 2005, pp.~283--309.

\bibitem{Sun13}
{\sc W.~Sun}, {\em On a class of fully nonlinear elliptic equations on closed
  {H}ermitian manifolds}, arXiv:1310.0362.

\bibitem{Tian97}
{\sc G.~Tian}, {\em K\"ahler-{E}instein metrics with positive scalar
  curvature}, Invent. Math., 137 (1997), pp.~1--37.

\bibitem{TWWY14}
{\sc V.~Tosatti, Y.~Wang, B.~Weinkove, and X.~Yang}, {\em $c^{2,alpha}$
  estimates for nonlinear elliptic equations in complex and almost complex
  geometry}, arXiv:1402.0554.

\bibitem{Tru95}
{\sc N.~Trudinger}, {\em On the {D}irichlet problem for {H}essian equations},
  Acta Math., 175 (1995), pp.~151--164.

\bibitem{Wei04}
{\sc B.~Weinkove}, {\em Convergence of the {J}-flow on {K}\"ahler surfaces},
  Comm. Anal. Geom., 12 (2004), pp.~949--965.

\bibitem{W06}
\leavevmode\vrule height 2pt depth -1.6pt width 23pt, {\em On the {J}-flow in
  higher dimensions and the lower boundedness of the {M}abuchi energy}, J.
  Differential Geom., 73 (2006), pp.~351--358.

\bibitem{Yao14}
{\sc Y.~Yao}, {\em The {J}-flow on toric manifolds}, arXiv:1407.1168.

\bibitem{Yau78}
{\sc S.-T. Yau}, {\em On the {R}icci curvature of a compact {K}\"ahler manifold
  and the complex {M}onge-{A}mp\`ere equation {I}.}, Comm. Pure Appl. Math., 31
  (1978), pp.~339--411.

\bibitem{Yau93}
\leavevmode\vrule height 2pt depth -1.6pt width 23pt, {\em Open problems in
  geometry}, Proc. Symposia Pure Math., 54 (1993), pp.~1--28.

\bibitem{Zheng14}
{\sc K.~Zheng}, {\em {$I$}-properness of {M}abuchi's {$K$}-energy},
  arXiv:1410.1821.

\end{thebibliography}

\end{document}